\documentclass[dvips,12pt,a4paper,twoside]{article}
\usepackage{authblk}
\usepackage{amsmath,amsfonts,euscript,amssymb,amsthm}

\usepackage{ifthen}
\usepackage{graphicx}
\usepackage{epsfig}
\usepackage{nicefrac}
\usepackage{mathrsfs}
\usepackage[a4paper,left=32mm,right=32mm,top=31mm,bottom=30mm]{geometry}

\usepackage{amsmath}
\usepackage{amsfonts}
\newcommand{\R}{\mathbb{R}}
\newcommand{\N}{\mathbb{N}}
\newcommand{\C}{\mathbb{C}}
\newcommand{\Z}{\mathbb{Z}}
\newcommand{\eps}{\varepsilon}

\newcommand{\ra}{\rangle}
\newcommand{\la}{\langle}

\newcommand{\del}{\partial}

\newcommand{\diam}{\mathrm{diam}}

\newcommand{\id}{\mathrm{id}}

\newcommand{\per}{\mathrm{per}}

\def\ri{{\rm i}}%

\renewcommand{\Re}{\mathrm{Re}}

\def\Xint#1{\mathchoice
   {\XXint\displaystyle\textstyle{#1}}%
   {\XXint\textstyle\scriptstyle{#1}}%
   {\XXint\scriptstyle\scriptscriptstyle{#1}}%
   {\XXint\scriptscriptstyle\scriptscriptstyle{#1}}%
   \!\int}
\def\XXint#1#2#3{{\setbox0=\hbox{$#1{#2#3}{\int}$}
     \vcenter{\hbox{$#2#3$}}\kern-.5\wd0}}
\def\meanint{\Xint-}

\usepackage[
    bookmarks,
    colorlinks=true,
    bookmarksopen=true,
    bookmarksopenlevel=1,
    linkcolor=black,
    citecolor=blue,
    urlcolor=blue,
    filecolor=blue,
    anchorcolor=red,
    pdfstartview=FitH,
    pdfpagelayout=OneColumn,
    plainpages=false, 
    pdfpagelabels, 
    hypertexnames=false, 
    linktocpage, 
]{hyperref}

\newtheorem{theorem}{Theorem}[section]

\newtheorem{lemma}[theorem]{Lemma}

\newtheorem{corollary}[theorem]{Corollary}
\newtheorem{assumption}[theorem]{Assumption}

\newtheorem{remark}[theorem]{Remark}

\numberwithin{equation}{section}

\title{Bloch-wave homogenization on large time scales and dispersive
  effective wave equations}

 \author{T.\,Dohnal, A.\,Lamacz, B.\,Schweizer\thanks{Technische
     Universit\"at Dortmund, Fakult\"at f\"ur Mathematik,
     Vogelpothsweg 87, D-44227 Dortmund, Germany.}}

\begin{document}

\maketitle

\begin{abstract}
  We investigate second order linear wave equations in periodic media,
  aiming at the derivation of effective equations in $\R^n$, $n \in
  \{1, 2, 3\}$. Standard homogenization theory provides, for the limit
  of a small periodicity length $\eps>0$, an effective second order
  wave equation that describes solutions on time intervals $[0,T]$. In
  order to approximate solutions on large time intervals
  $[0,T\eps^{-2}]$, one has to use a dispersive, higher order wave
  equation. In this work, we provide a well-posed, weakly dispersive
  effective equation, and an estimate for errors between the solution
  of the original heterogeneous problem and the solution of the
  dispersive wave equation.  We use Bloch-wave analysis to identify a
  family of relevant limit models and introduce an approach to select
  a well-posed effective model under symmetry assumptions on the
  periodic structure. The analytical results are confirmed and
  illustrated by numerical tests.
\end{abstract}



\pagestyle{myheadings} 
\thispagestyle{plain} 

\markboth{T.\,Dohnal, A.\,Lamacz, B.\,Schweizer}{Homogenization and
  dispersive effective wave equations}

\section{Introduction}

The wave equation describes wave propagation in very different
applications, ranging from elastic waves to electro-magnetic waves. In
some applications, it is of interest to describe waves in periodic
media, where the period $\eps>0$ is much smaller than the
wave-length. The most fundamental questions regard the effective wave
speed and the dispersive behavior due to the heterogeneities.

We concentrate on the simplest model, the second order wave equation
in divergence form. For notational convenience, we restrict ourselves
to a unit density coefficient and study, for $x\in \R^n$, the wave
equation
\begin{equation}
  \label{eq:eps-wave}
  \del_t^2 u^\eps(x,t) = \nabla\cdot (a^\eps(x) \nabla u^\eps(x,t))\,.
\end{equation}
The medium is characterized by a positive coefficient matrix $a^\eps
:\R^n \to \R^{n\times n}$.  We are interested in periodic media with a
small periodicity length-scale $\eps>0$, and assume that $a^\eps(x) =
a_Y(x/\eps)$, where $a_Y:\R^n\to \R^{n\times n}$ is periodic.  The
wave equation is complemented with the initial condition
\begin{equation}
  \label{eq:initial}
  u^\eps(x,0) = f(x), \quad \del_t u^\eps(x,0) = 0\,.
\end{equation}

\begin{assumption}\label{ass:a-f}
  On the initial data we assume that $f\in L^2(\R^n)\cap L^1(\R^n)$
  has the Fourier representation
  \begin{equation}
    \label{eq:f-cond}
    f(x) = \frac1{(2\pi)^{n/2}} \int_{\R^n} F_0(k)\, e^{+\ri k\cdot x}\, dk\,,
  \end{equation}
  where the function $F_0:\R^n\to \C$ is supported on the compact set
  $K\subset\subset \R^n$.

  On the coefficient $a_Y:\R^n\to \R^{n\times n}$ we assume
  $Y$-periodicity for the cube $Y := (-\pi, \pi)^n \subset \R^n$ and
  the regularity $a_Y\in C^1(\R^n,\R^{n\times n})$. Moreover, we
  assume that $a_Y(y)$ is a symmetric and positive definite matrix
  field: for some $\gamma>0$ there holds $(a_Y(y))_{ij}=(a_Y(y))_{ji}$
  for all $i,j\in\{1,...,n\}$ and $\sum_{i,j=1}^n
  (a_Y(y))_{ij}\xi_i\xi_j\geq\gamma |\xi|^2$ for every $y\in\R^n$ and
  all $\xi\in\R^n$.
		
  The set $Y\subset \R^n$, the reciprocal cell $Z := (-1/2,
  1/2)^n\subset \R^n$, and the support $K\subset \R^n$ are fixed data
  of the problem.
\end{assumption}

The Fourier transform is always understood in the sense of
$L^2(\R^n)$.  We note that $F_0$ is bounded because of $f\in
L^1(\R^n)$. Since $F_0$ has compact support, every derivative of $f$
is of class $L^2(\R^n)$, hence $f\in C^\infty(\R^n)$.  We will later
restrict ourselves to dimensions $n\le 3$, an assumption that is used
in Sobolev-embeddings. General dimensions can be treated under
stronger regularity assumptions on $a_Y$.

\smallskip The fundamental question of homogenization theory is the
following: For small $\eps>0$, can the solution $u^\eps$ be
approximated by a solution of an equation with constant coefficients?
The answer is affirmative: There exists an effective coefficient
matrix $A\in \R^{n\times n}$, computable from $a_Y$, such that the
following holds: on an arbitrary time interval $[0,T]$, if $w :
\R^n\times [0,T]\to \R$ is the solution of
\begin{equation}
  \label{eq:eps-wave-homogen}
  \del_t^2 w(x,t) = \nabla\cdot (A \nabla w(x,t))\,,\quad 
  w(x,0) = f(x), \quad \del_t w(x,0) = 0\,,
\end{equation}
there holds $u^\eps\to w$ as $\eps\to 0$. For the result and function
spaces see e.g.\,\cite{FrancfortMR1172450}.

We are interested in a refinement of this result. Our aim is to
investigate the behavior of solutions $u^\eps$ of \eqref {eq:eps-wave}
for large times, namely for all $t\in [0, T_0 \eps^{-2}]$ with
$T_0>0$. It is well-known that the homogenized equation \eqref
{eq:eps-wave-homogen} cannot provide an approximation of $u^\eps$ on
the interval $[0, T_0 \eps^{-2}]$. Instead, we need a dispersive
equation to approximate $u^\eps$.

\paragraph{Main result.} In addition to the coefficient matrix $A\in
\R^{n\times n}$, we will define $E\in \R^{n\times n}$ and $F\in
\R^{n\times n\times n\times n}$, computable from the coefficient
$a_Y(.)$ with the Bloch eigenvalue problem on the periodicity cell
$Y$. The constant coefficient matrices define linear spatial
differential operators: the two second order operators $A D^2 =
\sum_{i,j} A_{ij} \del_i \del_j$ and $E D^2 = \sum_{i,j} E_{ij} \del_i
\del_j$, and the fourth order operator $F D^4 = \sum_{i,j,m,l}
F_{ijml} \del_i \del_j \del_m \del_l$. The weakly dispersive effective
equation reads
\begin{equation}
  \label{eq:weakly-dispersive}
  \del_t^2 w^\eps = A D^2 w^\eps + \eps^2 E D^2 \del_t^2 w^\eps 
  - \eps^2 F D^4 w^\eps\,.
\end{equation}
As initial conditions we use once more $w^\eps(x,0) = f(x)$ and
$\del_t w^\eps(x,0) = 0$. Equation \eqref {eq:weakly-dispersive} is of
fourth order in the spatial variables, and it contains a term that
uses second spatial {\em and} second time derivatives. The operator
contains the small parameter $\eps>0$ explicitly. It can nevertheless
be regarded as an effective equation in the sense of homogenization
theory, since the coefficients are $x$-independent. Numerically,
\eqref {eq:weakly-dispersive} is much easier to solve than \eqref
{eq:eps-wave}, since the fine scale need not be resolved.  The
contributions of higher order (operators with factor $\eps^2$)
describe the (weak) dispersive effects due to the heterogeneity of the
medium. Formally, for $\eps=0$, we recover the homogenized equation
\eqref {eq:eps-wave-homogen}.

Our main result shows that the weakly dispersive equation \eqref
{eq:weakly-dispersive} provides, for large times, an approximation of
the original equation \eqref {eq:eps-wave}. To our knowledge, both
aspects of our theorem are new in dimension $n>1$: (i) the
specification of a well-posed weakly dispersive effective wave
equation and (ii) the rigorous proof of the homogenization error
estimate on large time scales.

\begin{theorem}\label{thm:main}
  Let $\eps = \eps_l \to 0$ be a sequence of positive numbers and $n
  \in \{1, 2, 3\}$ be the dimension. Let the medium $a_Y:\R^n\to
  \R^{n\times n}$ and the initial data $f: \R^n\to \R$ satisfy
  Assumption \ref {ass:a-f}. We assume that $y\mapsto a_Y(y)$ is
  symmetric under reflections $y_j \longleftrightarrow -y_j$, and
  symmetric under coordinate exchanges $y_j \longleftrightarrow y_k$,
  see \eqref {eq:sym_ass}.

  We use the coefficient matrices $A$ and $C$ defined in \eqref
  {eq:Taylor-mu}, $E$ and $F$ as defined in Lemma \ref
  {lem:decompose}.  Then the following holds:
  \begin{enumerate}
  \item {\bf Well-posedness} Equation \eqref {eq:weakly-dispersive}
    with initial condition \eqref {eq:initial} has a unique solution
    $w^\eps$ for all positive times (see Theorem \ref
    {thm:main-approximation} for function spaces).
  \item {\bf Error estimate} Let $w^\eps$ be the solution of \eqref
    {eq:weakly-dispersive}, and let $u^\eps$ be the solution of \eqref
    {eq:eps-wave} for the same initial condition \eqref
    {eq:initial}. Then, with a constant $C_0 = C_0(a_Y, T_0, f)$,
    there holds the error estimate
    \begin{equation}
      \label{eq:approx-result}
      \sup_{t\in [0,T_0 \eps^{-2}]} 
      \| u^\eps(.,t) - w^\eps(.,t) \|_{L^2(\R^n) + L^\infty(\R^n)} 
      \le C_0 \eps\,.
    \end{equation}
  \end{enumerate}
\end{theorem}
The definition of the norm in \eqref {eq:approx-result} is recalled at
the end of Section \ref {ssec.SanSym}.  The $L^2(\R^n)$-norm is a
result of the Bloch-wave expansion (it appears e.g.\,in Theorem
\ref{thm:SanSym-1}), while the $L^\infty(\R^n)$-norm appears in the
control of error terms after Theorem \ref {thm:SanSym-2}, but also in
energy estimates, see Lemma \ref {lem:interpolation}.

\subsection*{Comparison with the literature}

The derivation of effective equations in periodic homogenization
problems is an old subject \cite {Sanchez}, two-scale convergence
\cite {Allaire1992} is today the most relevant analytical tool. The
use of Bloch-wave expansions \cite {Wilcox} was explored only more
recently, see e.g.\,\cite {MR1897707, MR2219790, MR1484944}.

Compared to elliptic and parabolic equations, some distinctive
features are relevant in the analysis of the wave equation. One
observation of \cite {FrancfortMR1172450} was that convergence of
energies can only be expected for initial data that are adapted to the
periodic medium, see also \cite{FrancfortMurat1992}.  Diffraction and
dispersion effects are analyzed in the spirit of homogenization theory
in \cite{MR2060593, MR2533955}. While the underlying questions are
similar, these contributions study a different scaling behavior in
$\eps$.  Other homogenization results for the wave equation are
contained in \cite{ZuazuaMR1760033, FrancuKrejciMR1727713, Lebeau,
  OriveZuazua, MR2511805, SchweizerVeneroni-Periodic}.

The study of dispersive effects and the derivation of a dispersive
effective wave equation are central aims in the works of Chen, Fish,
and Nagai, e.g. \cite{ChenFishMR2097759, ChenFish-Uniformly,
  ChenFishMR1896977, ChenFishMR1896976}. The authors expand several
ideas to treat the problem, among others they propose to introduce a
slow and a fast time scale to capture the long-time behavior of waves.
The authors concentrate on numerical studies and do not provide a
derivation of an effective model.

\paragraph{Derivation of dispersive models.}
To our knowledge, the first rigorous result that establishes a
dispersive model for the wave equation in the scaling of \eqref
{eq:eps-wave} appeared in \cite {Lamacz-Disp}. In that contribution,
the one-dimensional case $n=1$ was analyzed, the one-dimensional
version of \eqref {eq:weakly-dispersive} was formulated (in this case,
$A$, $E$, and $F$ are scalar coefficients and the differential
operator is $D = \del_x$), and a result similar to our Theorem \ref
{thm:main} was shown: the well-posedness of the dispersive equation
and an error bound on large time intervals.

Beyond the one dimensional case, we are not aware of any rigorous
results. The most relevant contribution with the perspective taken
here is \cite {SanSym}. In that paper, Bloch-wave expansions are used
to analyze the problem, mathematical insight is gained, and the
dispersive wave equation \eqref {eq:formal-weakly-disp} is formulated
(not in one of the theorems, but as a formal consequence on page
992). We use many of the ideas of that contribution.

Equation \eqref {eq:formal-weakly-disp} appears also as equation (42)
in \cite{ChenFishMR1896976}, the authors call it the ``bad''
Boussinesq equation.  The problem about this equation is its
ill-posedness: Loosely speaking, the equation is of the form $\del_t^2
u + L u = 0$, with $L = -\Delta - \eps^2\Delta^2$. The lowest order part
(in $\eps$) of $L$ is $-\Delta$, hence a positive operator, but for
every $\eps>0$, the operator is negative, since $\Delta^2$ is positive
and contains the highest order of differentiation. One can speculate
that this was the reason why no effective dispersive models were
rigorously formulated in the above mentioned works.

It was already observed in \cite{ChenFishMR1896976}, that a ``good''
Boussinesq equation can be obtained with a simple trick: Going back to
the prototype problem $\del_t^2 u = -L u = \Delta u + \eps^2\Delta^2
u$, we replace $\Delta u$ to lowest order (in $\eps$) by $\del_t^2 u$
and write the equation as $\del_t^2 u = \Delta u +
\eps^2\Delta\del_t^2 u$. In this form, the equation is
well-posed. This observation was also exploited in \cite
{Lamacz-Disp}, where it was shown rigorously that the ``good''
Boussinesq equation is the effective model for large times in the
one-dimensional case.

In this contribution we treat the higher dimensional case, using
methods that are completely different from those of \cite
{Lamacz-Disp}. Our new results rely on a Bloch-wave expansion of the
solution $u^\eps$, which we analyze in Sections \ref
{ssec.Blochwaveexpansion}--\ref {ssec.ExpansionDispersion}; in this
part we follow closely the ideas of \cite {SanSym}.  To clearify the
connection to this well-known article, we repeat that no convergence
result appears in \cite {SanSym}, function spaces and assumptions are
not always clearly specified in \cite {SanSym}, and only the ``bad''
Boussinesq equation appears (with a wrong sign and without further
discussion) in \cite {SanSym}.

We have to introduce two assumptions: (i) inital data are compactly
supported in Fourier space and (ii) the heterogeneous medium has
certain symmetries in the cell $Y$.  Both assumptions can possibly be
relaxed with some additional effort and new decomposition techniques;
our aim here is to present the long-time homogenization result in the
simplest relevant case. Due to the multi-dimensional setting, we have
anyway to work with tensors of coefficients to transform the ``bad''
effective equation into the ``good'' one. We show with mathematical
rigor that the weakly dispersive effective equation has the
approximation property for large times.

\smallskip In Section \ref{sec.Bloch} we expand the solution $u^\eps$
in Bloch waves, in Section \ref{sec.wellposed} we analyze the weakly
dispersive equation \eqref {eq:weakly-dispersive}.  The proof of
Theorem \ref {thm:main} is concluded at the end of Section
\ref{sec.wellposed}. Section \ref {sec.numerics} contains numerical
results.

\section{Approximation with a Bloch wave expansion}
\label{sec.Bloch}

In this section we present, in slightly changed notation and with
mathematical rigor regarding assumptions and norms, the approximation
results of \cite {SanSym}. To simplify some of the notation of \cite
{SanSym}, we consider here only the mass-density $\bar\rho \equiv 1$
and the scaling factor $\lambda = 1$.  

\subsection{Bloch wave expansion}
\label{ssec.Blochwaveexpansion}

We are given a periodic medium by the coefficient matrix $a_Y(y)$ on
the cube $Y$. The Bloch wave expansion uses functions $\psi_m$, which
are solutions of a periodic eigenvalue problem on $Y$.  The wave
parameter $k$ is a vector in the reciprocal periodicity cell $Z =
(-1/2, 1/2)^n$.  At this point, we regard $k\in Z$ as a given
parameter and consider
\begin{equation}
  \label{eq:eigenvalue}
  - (\nabla_y + \ri k ) \cdot (a_Y(y) (\nabla_y + \ri k ) \psi_m(y,k) ) 
  = \mu_m(k) \psi_m(y,k)\,.
\end{equation}
We search for $\psi_m(.,k) : Y \to \C$ in the space $H^1_\per(Y)$,
defined as the space of periodic functions on $Y$ of class $H^1$.  We
find a family (indexed by $m\in \N=\{0,1,2,\dots\}$) of periodic
solutions $\psi_m(.,k) : Y \to \C$ with non-negative real eigenvalues
$\mu_m(k)$, $\mu_{m+1}(k) \ge \mu_m(k)$, both the solution and the
eigenvalue depend on $k$. We assume that the functions are normalized
in $L^2(Y)$, $\| \psi_m \|_{L^2(Y)} = 1$.  Regarding the regularity of
$\psi_m$ we note that, for $a_Y$ of class $C^1$, standard elliptic
regularity theory implies $\psi_m\in H^2(Y)$.

Based on the eigenfunction $\psi_m$, we can construct the
quasi-periodic Bloch-waves $w_m(y,k) := \psi_m(y,k) e^{\ri k\cdot y}$,
which satisfy
\begin{equation}
  \label{eq:wm-problem}
  -\nabla_y\cdot (a_Y(y) \nabla_y w_m(y,k)) = \mu_m(k) w_m(y,k).
\end{equation}

We recall an essential fact regarding the completeness of these
eigenfunctions (see e.g.  \cite{MR1484944} for this well-known
result). The Bloch waves form a basis of $L^2(\R^n)$ in the sense that
every function $g\in L^2(\R^n)$ can be expanded as
\begin{equation}
  \label{eq:Bloch-complete}
  g(y) = \sum_{m=0}^\infty \int_Z \hat g_m(k) w_m(y,k)\, dk\,,\qquad
  \hat g_m(k) = \int_{\R^n} g(y) w_m(y,k)^*\, dy\,,
\end{equation}
where we use the star $^*$ to denote complex conjugation and the first
equality is understood in the sense of $L^2(\R^n)$-convergence of
partial sums.  There holds the Parseval identity
\begin{equation}
  \label{eq:Parseval}
  \| g \|_{L^2(\R^n)}^2 = \int_{\R^n} |g(x)|^2\, dx
  = \sum_{m=0}^\infty \int_Z |\hat g_m(k)|^2\, dk
  = \| \hat g \|_{l^2(\N, L^2(Z))}^2\,.
\end{equation}

\subsubsection*{Rescaled Bloch wave expansion}

We investigate a strongly heterogeneous medium $a^\eps(x) =
a_Y(x/\eps)$. Starting from the Bloch waves on the cube $Y$, we define
rescaled quantities as
\begin{align}
  &\psi_m^\eps(x,k) := \psi_m\left( \frac{x}{\eps} , \eps k \right)\,,\quad
  \mu_m^\eps(k)   := \frac1{\eps^2} \mu_m(\eps k)\,,\\
  &w_m^\eps(x,k) := w_m\left( \frac{x}{\eps} , \eps k \right) 
  = \psi_m^\eps(x,k) e^{\ri k\cdot x}
  = \psi_m\left( \frac{x}{\eps} , \eps k \right) e^{\ri k\cdot x}\,.
\end{align}
This choice guarantees, in particular,
\begin{equation}
  \label{eq:wm-eps-problem}
  -\nabla\cdot (a^\eps(x) \nabla w_m^\eps(x,k)) = \mu_m^\eps(k) w_m^\eps(x,k).
\end{equation}

The expansion formula \eqref {eq:Bloch-complete} in Bloch
eigenfunctions can be expressed in the new variables. Every function
$f\in L^2(\R^n)$ can be written as
\begin{equation}
  \label{eq:Bloch-complete-eps}
  f(x) = \sum_{m=0}^\infty \int_{Z/\eps} \hat f_m^\eps(k) w_m^\eps(x,k)\, dk\,,\qquad
  \hat f_m^\eps(k) = \int_{\R^n} f(x) w_m^\eps(x,k)^*\, dx\,.
\end{equation}
To verify this formula, it suffices to set $f(x) = g(x/\eps)$ and
$\hat f_m^\eps(k) = \eps^n \hat g_m(\eps k)$ and to use \eqref
{eq:Bloch-complete}. This shows additionally the Parseval identity in
transformed variables,
\begin{equation}
  \label{eq:Parseval-eps}
  \| f \|_{L^2(\R^n)}^2 = \int_{\R^n} |f(x)|^2\, dx
  = \sum_{m=0}^\infty \int_{Z/\eps} |\hat f_m^\eps(k)|^2\, dk
  = \| \hat f^\eps \|_{l^2(\N, L^2(Z/\eps))}^2\,.
\end{equation}

In our situation of $a_Y\in C^1(\R^n,\R^{n\times n})$ and $f\in
H^2(\R^n)$, the series in \eqref {eq:Bloch-complete-eps} is also
convergent in $H^1(\R^n)$. We provide a proof in Appendix \ref
{app.H1-convergence}.

\subsubsection*{Expansion of the solution}
The Bloch-wave formalism can provide a formula for the solution of the
original wave equation.

\begin{lemma}[Expansion of the solution]\label{lem:sol-expansion}
  Let the medium $a_Y:\R^n\to \R^{n\times n}$ and the initial data $f:
  \R^n\to \R$ satisfy Assumption \ref {ass:a-f}.  Then, for every
  $\eps>0$ and every $T_\eps\in (0,\infty)$, the wave equation \eqref
  {eq:eps-wave} has a unique weak solution $u^\eps$ with the
  regularity $u^\eps(x,t)\in L^\infty(0,T_\eps; H^2(\R^n)) \cap
  W^{1,\infty}(0,T_\eps; H^1(\R^n)) \cap W^{2,\infty}(0,T_\eps;
  L^2(\R^n))$.

  The solution $u^\eps$ of \eqref {eq:eps-wave} can be represented as
  \begin{equation}
    \label{eq:Bloch-complete-soln-eps}
    u^\eps(x,t) 
    = \sum_{m=0}^\infty \int_{Z/\eps} \hat f_m^\eps(k)\, w_m^\eps(x,k)
    \, \Re\left(e^{\ri t \sqrt{\mu_m^\eps(k)}}\right)\, dk\,.
  \end{equation}
  Here, the right hand side is understood as the strong
  $L^2(\R^n)$-limit of partial sums, for every fixed $t\ge 0$, and
  $\Re(.)$ denotes the real part.
\end{lemma}

Before we start the proof, we note that the expression in \eqref
{eq:Bloch-complete-soln-eps} formally defines a solution of \eqref
{eq:eps-wave}--\eqref {eq:f-cond}. In fact, the second time derivative
of the right hand side is given by the same formula, introducing only
the additional factor $-\mu_m^\eps(k)$ under the integral. On the
other hand, the application of the operator $\nabla\cdot (a^\eps(x)
\nabla)$ to the integrand produces, by \eqref {eq:wm-eps-problem}, the
same result.

\begin{proof}
  {\em Step 1. The weak solution.} A weak solution $u^\eps$ can be
  constructed, e.g., with a Galerkin scheme. One exploits the energy
  estimate which is obtained with a multiplication of equation \eqref
  {eq:eps-wave} by the real function $\del_t u^\eps$,
  \begin{align*}
    0 = \int_{\R^n} [\del_t^2 u^\eps(.,t) - \nabla\cdot (a^\eps(.) \nabla
    u^\eps(.,t))] \del_t u^\eps
    = \frac{d}{dt} \frac12 \int_{\R^n} |\del_t u^\eps(t) |^2 
    +  |\nabla  u^\eps(t)|_{a^\eps}^2 \,,
  \end{align*}
  where the last equality holds, since $a^\eps(x)$ is a symmetric
  matrix for every $x\in\R^n$.  Here and below we use the notation
  $|\xi|^2_a := \xi^*\cdot (a\cdot \xi)$ for vectors $\xi\in\C^n$ and
  matrices $a\in\R^{n\times n}$.  Also higher order estimates can be
  obtained. We use $L^\eps := \nabla\cdot (a^\eps(x) \nabla)$ and
  multiply the equation $\del_t^2 u^\eps = L^\eps u^\eps$ by $-\del_t
  (L^\eps u^\eps)$ to find
  \begin{equation}\label{eq:higher-order}
    \frac{d}{dt} \frac12 \int |\del_t\nabla u^\eps|_{a^\eps}^2 
    +  |L^\eps u^\eps|^2  = 0\,.
  \end{equation}
  Since the initial data are $u|_{t=0} = f \in H^2(\R^n)$ and $\del_t
  u|_{t=0} = 0$, we obtain estimates for $u^\eps$ in the function
  spaces that are stated in the Theorem. The estimates for $L^\eps
  u^\eps(.,t) = \del_t^2 u^\eps(.,t) \in L^2(\R^n)$ imply the
  regularity $u^\eps \in W^{2,\infty}(0,T_\eps; L^2(\R^n))$ and the
  estimates for $u^\eps(.,t)\in H^2(\R^n)$ due to $a_Y\in
  C^1(Y,\R^{n\times n})$ by standard elliptic regularity
  theory. Uniqueness within the given class follows from linearity,
  repeating the above calculations for differences of solutions.

  \smallskip {\em Step 2. Convergence in \eqref
    {eq:Bloch-complete-soln-eps}.}  The Parseval identity \eqref
  {eq:Parseval-eps} implies that the coefficient functions define an
  element $(\hat f^\eps_m(k))_{m,k}$ of $l^2(\N, L^2(Z/\eps))$. As a
  consequence, also the modified coefficients $\left(\hat f^\eps_m(k)
    \Re\left( e^{\ri t \sqrt{\mu_m^\eps(k)}}\right)\right)_{m,k}$
  define an element in the same space, since all factors have absolute
  value bounded by $1$. Using again the Parseval identity \eqref
  {eq:Parseval-eps}, we conclude that the sum of \eqref
  {eq:Bloch-complete-soln-eps} converges in $L^2(\R^n)$, independently
  of $t\ge 0$.

  \smallskip {\em Step 3. Identification of $u^\eps$.}  We consider a
  partial sum $\sum_{m=1}^M$ in \eqref {eq:Bloch-complete-soln-eps} to
  define a function $u^\eps_M$ and observe that this provides a strong
  solution $u^\eps_M$ of the wave equation to the initial values $f_M
  = \sum_{m=0}^M \int_{Z/\eps} \hat f_m^\eps(k) w_m^\eps(x,k)\, dk$
  and vanishing initial velocity. This fact can be checked with a
  direct calculation: the operator $\nabla\cdot (a^\eps(x) \nabla)$ is
  understood in the weak form and can be applied to the
  $H^1(Y)$-functions $w_m^\eps$. We claim that $u^\eps_M$ forms a
  Cauchy sequence in the space $L^\infty([0,T_\eps], H^1(\R^n))$. This
  follows with a testing argument, exploiting
  \begin{align*}
   \int_{\R^n}| \nabla u^\eps_M(t) - \nabla u^\eps_N(t) |_{a^\eps}^2
   + | \del_t u^\eps_M(t) - \del_t u^\eps_N(t) |^2
   =    \int_{\R^n}| \nabla f^\eps_M - \nabla f^\eps_N |_{a^\eps}^2
   \to 0
  \end{align*}
  for $M, N \to \infty$ due to the $H^1(\R^n)$-convergence in \eqref
  {eq:Bloch-complete-eps}. We conclude that $u^\eps_M$ converges to a
  limit function. The limit function is again a weak solution of the
  wave equation, from the uniqueness of weak solutions we conclude
  $u^\eps_M \to u^\eps$ for $M\to \infty$.

  On the other hand, as observed in Step 2, by definition of
  $u^\eps_M$, the limit function is given by the right hand side of
  \eqref {eq:Bloch-complete-soln-eps}.
\end{proof}

\subsection{The approximation results of Santosa and Symes}
\label{ssec.SanSym}

With the next two theorems we observe that, for small $\eps>0$, the
expression of \eqref {eq:Bloch-complete-soln-eps} may be
simplified. In our first simplification we realize that all indices
$m$ with $m\ge 1$ can be neglected. This observation is a fundamental
tool in the Bloch-wave homogenization method and is also used, e.g.,
in \cite {MR1642454, MR1897707, MR1484944}.

\begin{theorem}[Santosa and Symes \cite {SanSym}, Theorem
  1]\label{thm:SanSym-1}
  Let the medium $a_Y:\R^n\to \R^{n\times n}$ and the initial data $f:
  \R^n\to \R$ satisfy Assumption \ref {ass:a-f}.  Let $u^\eps :
  [0,\infty) \to H^2(\R^n)$ be given by \eqref
  {eq:Bloch-complete-soln-eps}.  Then there exists $C = C(f) > 0$ such
  that
  \begin{equation}
    \label{eq:thm1-SanSym}
    \sup_{t\in (0,\infty)} \left\| \sum_{m=1}^\infty 
      \int_{Z/\eps} \hat f_m^\eps(k)\, w_m^\eps(x,k)\, \Re
      \left(e^{\ri t \sqrt{\mu_m^\eps(k)}}\right)\, dk 
    \right\|_{L^2(\R^n)} \le C \eps\,.
  \end{equation}
\end{theorem}

\begin{proof}
  We consider a single coefficient $\hat f_m^\eps(k) \Re\,
  \left(e^{\ri t \sqrt{\mu_m^\eps(k)}}\right)$ in the expansion of
  $u^\eps$ in \eqref {eq:Bloch-complete-soln-eps}. We use first the
  inversion formula \eqref {eq:Bloch-complete-eps} to evaluate this
  coefficient, then the eigenvalue property \eqref {eq:wm-eps-problem}
  to introduce the factor $\mu_m^\eps(k) = \eps^{-2} \mu_m(\eps k)$,
  then integration by parts and the solution property of $u^\eps$,
  \begin{align}
    \hat f_m^\eps(k)\, \Re\left(e^{\ri t  \sqrt{\mu_m^\eps(k)}}\right)
    &= \int_{\R^n} u^\eps(x,t) w_m^\eps(x,k)^*\, dx \notag\\
    & = -\frac{1}{\mu_m^\eps(k)} \int_{\R^n} u^\eps(x,t) 
    [\nabla\cdot (a^\eps(x) \nabla w_m^\eps(x,k))]^* \, dx \notag\\
    & = -\frac{\eps^2}{\mu_m(\eps k)} \int_{\R^n} [\del_t^2 u^\eps(x,t)]
    w_m^\eps(x,k)^* \, dx\,.\label{eq:dt_sq_bloch}
  \end{align}

  We claim that, with $C>0$ independent of $t\in [0,\infty)$, the
  functions $x\mapsto \del_t^2 u^\eps(x,t)$ satisfy the estimate $\|
  \del_t^2 u^\eps(.,t) \|_{L^2(\R^n)} \le C \eps^{-1}$.  Indeed, this
  bound can be obtained as in \eqref {eq:higher-order}, where
  multiplication of $\del_t^2 u^\eps = L^\eps u^\eps$ with $\del_t
  L^\eps u^\eps$ provided
  \begin{align*}
    \int_{\R^n} |\del_t\nabla u^\eps(.,t)|^2_{a^\eps} + |L^\eps
    u^\eps(.,t)|^2 = \int_{\R^n} |L^\eps u^\eps(.,0)|^2 \,.
  \end{align*}
  Since the initial data $f$ are smooth, we have $\| L^\eps
  u^\eps|_{t=0} \|_{L^2(\R^n)} = \| \nabla\cdot (a^\eps(x) \nabla
  f)\|_{L^2(\R^n)} \le C \eps^{-1}$, hence $\| L^\eps
  u^\eps(.,t)\|_{L^2(\R^n)} \leq C\eps^{-1}$.  Accordingly, by the
  evolution equation, we also have $\| \del_t^2 u^\eps(.,t)
  \|_{L^2(\R^n)} = \| L^\eps u^\eps(.,t)\|_{L^2(\R^n)} \le C
  \eps^{-1}$.

  We can now continue \eqref{eq:dt_sq_bloch}. From the Parseval
  identity \eqref {eq:Parseval-eps} we obtain
  \begin{align*}
    \left\| \mu_m(\eps k) \hat f_m^\eps(k)\, \Re \left(e^{\ri t
        \sqrt{\mu_m^\eps(k)}}\right) \right\|_{l^2(\N, L^2(Z/\eps))} = \eps^2
    \| \del_t^2 u^\eps(.,t) \|_{L^2(\R^n)} \le C \eps\,.
  \end{align*}
  It remains to observe that omitting the term $m=0$ decreases the
  norm on the left hand side of this relation. Regarding terms with
  $m\ge 1$, we exploit that there exists a lower bound $c_0 > 0$ such
  that eigenvalues are bounded from below, $\mu_m(\xi) \ge c_0$,
  independent of $\xi\in Z$ and $m\ge 1$,
  cf. \cite{MR1484944}. Another application of the Parseval identity
  provides the claim \eqref {eq:thm1-SanSym}.
\end{proof}

At this point, we have obtained a first approximation of the solution
$u^\eps$. In the expansion of $u^\eps$, all contributions from indices
$m\ge 1$ are not relevant at the lowest order (uniformly in
time). Theorem \ref {thm:SanSym-1} provides $\| u^\eps - u_0^\eps
\|_{L^\infty((0,\infty), L^2(\R^n))} \le C \eps$, where
\begin{equation}
  \label{eq:Bloch-complete-soln-eps-m0}
  u_0^\eps(x,t) :=  \int_{Z/\eps} \hat f_0^\eps(k) w_0^\eps(x,k)
  \, \Re\left( e^{\ri t \sqrt{\mu_0^\eps(k)}}\right)\, dk\,.
\end{equation}

We will now analyze $u_0^\eps$ further.  The next aim is to replace
the Bloch coefficient $\hat f_0^\eps(k)$ by the Fourier coefficient
$F_0(k)$. At this point, we make more substantial changes with respect
to \cite{SanSym}, where (without providing norms), the essence of the
subsequent results is observed in Theorem 2.

We start with a general observation regarding Fourier-transforms.
\begin{lemma}[Products with periodic functions]
  \label{lem:periodic-functions-product}
  Let $h\in L^2(\R^n,\C)\cap L^1(\R^n,\C)$ be a function in space
  dimension $n\le 3$. For fixed $\eps>0$, let $Y_\eps := (-\eps\pi,
  \eps\pi)^n$ be a periodicity cell, let $\Phi:\R^n\to \C$ be a
  $Y_\eps$-periodic function with $\Phi \in H^2_\per(Y_\eps,\C)$.

  If the Fourier transform of $h$ vanishes in grid points $\Z^n/\eps$,
  then its $L^2(\R^n)$-product with $\Phi$ vanishes. More precisely,
  there holds
  \begin{equation}
    \label{eq:product-lemma}
    \int_{\R^n} h(x) e^{\ri l\cdot x/\eps}\, dx = 0 
    \quad\forall\, l\in \Z^n    \qquad\Rightarrow\qquad 
    \int_{\R^n} h(x) \Phi(x)\, dx = 0.
  \end{equation}
\end{lemma}

\begin{proof} Without loss of generality, we consider only $\eps=1$
  and use $Y = Y_1$ in this proof. We expand the $L^2(Y)$-function
  $\Phi$ in a strongly $L^2(Y)$-convergent Fourier series
  \begin{align*}
    \Phi(x) = \sum_{l\in \Z^n} \alpha_l\, e^{\ri l\cdot x}\,\quad
    \text{ with }\quad (\alpha_l)_l \in l^2(\Z^n, \C)\,.
  \end{align*}
  Because of the regularity $\Phi\in H^2(Y)$, we have additionally the
  decay property $(|l|^2 \alpha_l )_l \in l^2(\Z^n,\C)$. In
  particular, because of $(|l|^{-2})_l \in l^2(\Z^n,\C)$ for $n\le 3$,
  the sequence of Fourier coefficients satisfies $(\alpha_l)_l\in
  l^1(\Z^n,\C)$.

  Since $h$ is of class $L^1(\R^n)$, we can approximate the integral
  on the right hand side of \eqref {eq:product-lemma} by integrals
  over large balls. For $R>0$, we use the ball $B_R(0)\subset
  \R^n$. Because of the embedding $H^2(Y)\subset L^\infty(Y)$ for
  $n\le 3$, the function $\Phi$ is bounded on $Y$. We can therefore
  write with an error term $\rho_1(R)\in \C$ satisfying $\rho_1(R)\to
  0$ for $R\to \infty$,
  \begin{align*}
    &\int_{\R^n} h(x) \Phi(x)\, dx =
    \int_{B_R(0)} h(x) \Phi(x)\, dx + \rho_1(R)\\
    &\qquad = \lim_{L\to \infty}  \sum_{l\in \Z^n, |l|\le L} \alpha_l
    \int_{B_R(0)} h(x) e^{\ri l\cdot x}\, dx + \rho_1(R)\\
    &\qquad = \lim_{L\to \infty}  \sum_{l\in \Z^n, |l|\le L} \alpha_l
    \int_{\R^n} h(x) e^{\ri l\cdot x}\, dx
    + \rho_2(R) + \rho_1(R)
    = \rho_2(R) + \rho_1(R)\,.
  \end{align*}
  In the second equality, we used $h\in L^2(\R^n)$ and the
  $L^2(B_R(0))$-convergence of the Fourier-series. In the fourth
  equality we exploited the assumption, which provides that each of
  the integrals vanishes. In the third equality, we introduced the
  error term $\rho_2(R)$, which satisfies
  \begin{align*}
    |\rho_2(R)| \le \lim_{L\to \infty}  \sum_{l\in \Z^n, |l|\le L} |\alpha_l| 
    \int_{\R^n\setminus B_R(0)} |h(x)|\, dx \to 0
  \end{align*}
  for $R\to \infty$ because of $h\in L^1(\R^n)$ and $(\alpha_l)_l\in
  l^1(\Z^n)$.  Since $R$ was arbitrary, the claim \eqref
  {eq:product-lemma} is verified.
\end{proof}

After this preparation, we can now prove that the Fourier transform
$F_0$ of $f$ is a good approximation of the Bloch wave coefficients
$\hat f_0^\eps$.
\begin{theorem}\label{thm:SanSym-2}
  Let the medium $a_Y:\R^n\to \R^{n\times n}$ and the initial data $f:
  \R^n\to \R$ satisfy Assumption \ref {ass:a-f}, let the dimension be
  $n\in \{1, 2, 3\}$.  Then, with $C = C(f) >0$, there holds
  \begin{equation}
    \label{eq:thm2-SanSym}
    \left\| \hat f_0^\eps - F_0 \right\|_{L^1(Z/\eps)} \le C\eps\,.
  \end{equation}
  Furthermore, for $0<\eps\le 1$ small enough to have $K\subset
  Z/\eps$, there holds
  \begin{equation}
    \label{eq:thm2-SanSym-support}
    \hat f_0^\eps(k) = 0 \quad \forall k\in (Z/\eps) \setminus K\,.
  \end{equation}
\end{theorem}

\begin{proof}
  {\em Step 1: $k\in K$.}  The difference of the two functions in
  \eqref {eq:thm2-SanSym} reads (for arbitrary $k\in K$)
  \begin{align*}
    \hat f_0^\eps(k) - F_0(k)
    = \int_{\R^n} f(x) e^{-\ri k\cdot x} 
    \left[ \psi_0\left( \frac{x}{\eps} , \eps k\right)^* 
      - \frac{1}{\sqrt{|Y|}} \right]\, dx\,.
   \end{align*}
   The periodic solution $\psi_0(.,0)$ to the wave vector $k=0$ is
   constant, by our normalization it is given as $\psi_0(y,0) =
   \sqrt{|Y|}^{-1}$ for every $y\in Y$.  Since $k$ ranges (in this
   step of the proof) in the bounded compact set $K$, we find the
   estimate
   \begin{equation}\label{eq:thm2-SanSym-b-4261}
     \sup_{k\in K} \sup_{x\in \R^n} 
     \left| \psi_0\left( \frac{x}{\eps} , \eps k\right)^* 
       - \frac{1}{\sqrt{|Y|}} \right| \le C\eps\,,
   \end{equation}
   for some constant $C = C(a_Y)$. This can be verified by writing the
   elliptic equation that is satisfied by the difference of the two
   solutions $\psi_0(.,\eps k)$ and $\psi_0(.,0) \equiv
   \sqrt{|Y|}^{-1}$. By elliptic regularity theory, the difference is
   of order $\eps$ in the norm $H^2(Y)$, which embeds continuously
   into $L^\infty(Y)$ (at this point we exploit $a_Y\in C^1$ to
   conclude the $H^2(Y)$-regularity and the assumption $n\le 3$ for
   the Sobolev embedding). Because of $f \in L^1(\R^n)$ we obtain
   \begin{align*}
     \left|  \hat f_0^\eps(k) - F_0(k) \right|
     \le C\eps  \| f \|_{L^1(\R^n)} \le C\eps\,,
     \label{eq:thm2-SanSym-b-5278}
   \end{align*}
   uniformly in $k\in K$. Since $K$ is compact, this provides also an
   $L^1(K)$-bound as in the statement of \eqref {eq:thm2-SanSym}.

   \smallskip {\em Step 2: $k\in (Z/\eps) \setminus K$.} The numbers
   $\eps\in (0,1]$ with $K\subset Z/\eps$ and $k\in (Z/\eps) \setminus
   K$ are kept fixed in the sequel. Our proof uses Lemma \ref
   {lem:periodic-functions-product} with the two functions $\Phi(x) :=
   \psi_0(x/\eps, \eps k)^*$ and $h(x) = f(x) e^{-\ri k\cdot
     x}$. These functions have the required regularities: $\Phi \in
   H^2_\per(Y_\eps)$ and $h\in L^2(\R^n)\cap L^1(\R^n)$.

   Regarding the Fourier transform of $h$ in grid-points $l/\eps\in
   \Z^n/\eps$ we calculate
   \begin{align*}
     \int_{\R^n} h(x) e^{\ri l\cdot x/\eps}\, dx 
     = \int_{\R^n} f(x) e^{-\ri k\cdot x} e^{\ri l\cdot x/\eps}\, dx 
     = (2\pi)^{n/2} F_0(k - (l/\eps))
     = 0\,.
   \end{align*}
   In the last step we exploited the fact that $k - (l/\eps) \not\in
   K$. This is obtained by a distinction of cases: For $l=0$, we have
   $k - (l/\eps) = k$, and we considered $k\not\in K$.  For $\Z^n \ni
   l \neq 0$, the number $k - (l/\eps)$ is outside $Z/\eps$ because of
   $k\in Z/\eps$. We obtain that the assumption of \eqref
   {eq:product-lemma} is satisfied.  The implication \eqref
   {eq:product-lemma} therefore implies
   \begin{equation}\label{eq:almost-there-Thm2}
     \hat f_0^\eps(k)
     = \int_{\R^n} f(x) e^{-\ri k\cdot x} 
     \psi_0\left( \frac{x}{\eps} , \eps k\right)^*
     = \int_{\R^n} h(x) \Phi(x)\, dx = 0\,.
   \end{equation}
   This verifies the claim \eqref {eq:thm2-SanSym-support} about the
   support of $\hat f_0^\eps$.  In turn, since both functions vanish
   outside $K$, it also implies the $L^1$-estimate \eqref
   {eq:thm2-SanSym} for the difference on all of $Z/\eps$.
\end{proof}

We use Theorem \ref {thm:SanSym-2} to simplify the representation of
$u_0^\eps$ of \eqref {eq:Bloch-complete-soln-eps-m0}.  We define a new
approximation as
\begin{equation}
  \label{eq:U-eps}
  \begin{split}
    U^\eps(x,t) := (2\pi)^{-n/2} \int_{K} F_0(k) e^{\ri k\cdot x}\, 
    \Re \left(e^{\ri t \sqrt{\mu_0^\eps(k)}}\right)\, dk\,.
  \end{split}
\end{equation}
Theorem \ref {thm:SanSym-2} allows to calculate, using once more
\eqref {eq:thm2-SanSym-b-4261} to compare $w_0^\eps(x,k) =
\psi_0(x/\eps, \eps k) e^{\ri k\cdot x}$ with $(2\pi)^{-n/2} e^{\ri
  k\cdot x}$,
\begin{align*}
  &\| u_0^\eps - U^\eps \|_{L^\infty((0,\infty)\times \R^n)}
  = \left\| \int_{K} \hat f_0^\eps(k) w_0^\eps(x,k)\, 
    \Re\left( e^{\ri t \sqrt{\mu_0^\eps(k)}}\right)\, dk
    - U^\eps \right\|_{L^\infty((0,\infty)\times \R^n)}\displaybreak[2]\\
  &\qquad\qquad\le \frac1{(2\pi)^{n/2}} \sup_{t\in (0,\infty)} \sup_{x\in \R^n}  \left|  
    \int_{K} \hat f_0^\eps(k) e^{\ri k\cdot x}\,  
    \Re\left( e^{\ri t \sqrt{\mu_0^\eps(k)}}\right)\, dk \right.\\
  &\qquad\qquad\qquad\qquad\qquad\qquad\qquad\qquad \left.   
    - \int_{K} F_0(k) e^{\ri k\cdot x} 
    \Re\left( e^{\ri t \sqrt{\mu_0^\eps(k)}}\right)\, dk
  \right| + C\eps\displaybreak[2]\\
  &\qquad\qquad\le C \left\| \hat f_0^\eps - F_0 \right\|_{L^1(Z/\eps)} 
  + C\eps \le C \eps.
\end{align*}
Due to the uniform error estimate in \eqref {eq:thm2-SanSym-b-4261},
the constant $C$ in the error term depends only on the norm $\| \hat
f_0^\eps(.) \|_{L^1(\R^n)}$.

We can combine this error estimate with the one obtained earlier for
the difference $\| u^\eps - u_0^\eps \|_{L^\infty((0,\infty),
  L^2(\R^n))}$. We use, given two norms $\| . \|_X$ and $\| . \|_Y$,
the new norm (weaker than both original norms) $\| u \|_{X+Y} :=
\inf\{\|u_1\|_X + \|u_2\|_Y : u = u_1 + u_2\}$. This allows to write
the combined estimate as 
\begin{equation}
  \label{eq:combined-u-U}
  \| u^\eps - U^\eps \|_{L^\infty((0,\infty),
  (L^\infty + L^2)(\R^n))} \le C\eps\,.  
\end{equation}

\subsection{Expansion of the dispersion relation}
\label{ssec.ExpansionDispersion}

The next step is to replace the eigenvalue $\mu_0$ by its Taylor
series. We note that in a neighborhood of $k=0$ the eigenvalue $\mu_0$
depends analytically on $k$ with $\mu_0(0)=\nabla \mu_0(0)=0$,
cf. \cite{MR1484944}. We denote the derivatives of $\mu_0$ as
$A_{lm}=\tfrac{1}{2}\partial_{k_l}\partial_{k_m}\mu_0(0)$,
$B_{lmn}=\tfrac{1}{6}\partial_{k_l}\partial_{k_m}\partial_{k_n}\mu_0(0)$,
and $C_{lmnq}=\tfrac{1}{24} \partial_{k_l}\partial_{k_m}\partial_{k_n}
\partial_{k_q}\mu_0(0)$.  The reflection symmetry $\mu_0(k)=\mu_0(-k)$
(valid without any structural assumptions on $a_Y$) provides that all
odd derivatives of $\mu_0$ vanish in $k=0$, see Remark \ref{rem:odd}
below. In particular, there holds $B=0$.  The Taylor series of $\mu_0$
in $k$ around $k = 0$ is therefore given as
\begin{equation}
  \label{eq:Taylor-mu}
  \mu_0(k)  =  \sum A_{lm} k_l k_m + \sum C_{lmnq} k_l k_m k_n k_q + O(|k|^6).
\end{equation}
Here and below, a bare sum is always over the repeated indices. The
expansion corresponds to an expansion of $\mu_0^\eps(k)$,
\begin{equation}
  \label{eq:Taylor-mu-eps}
  \mu_0^\eps(k)  = \frac1{\eps^2}  \mu_0(\eps k)  
  =  \sum A_{lm} k_l k_m + \eps^2 \sum C_{lmnq} k_l k_m k_n k_q + O(\eps^4)\,,
\end{equation}
the error is of order $\eps^4$, uniformly in $k\in K$.

In the spirit of this expansion, we next want to simplify further
$U^\eps$ of \eqref {eq:U-eps}.  We use $\Re(z) = \tfrac12 (z + z^*)$
and the Taylor expansion of the square root
\begin{equation}\label{eq:square-root-Taylor}
  \sqrt{a+c} = \sqrt{a} + \frac1{2\sqrt{a}} c + O(|c|^2)
\end{equation}
for $a\in \C\setminus \{0\}$ and $c\in \C$ with small absolute value.
We define $v^\eps$ (compare page 992 of \cite{SanSym}) as
\begin{equation}
  \label{eq:v-eps}
  \begin{split}
    v^\eps(x,t) :=  (2\pi)^{-n/2} \frac12 \sum_{\pm} \int_{K} F_0(k) 
    e^{\ri k\cdot x}
    \exp\left( \pm \ri t \sqrt{\sum A_{lm} k_l k_m}  \right)\qquad\qquad\\
    \times\   \exp\left( \pm \frac{\ri \eps^2}{2} t
      \frac{\sum C_{lmnq} k_l k_m k_n k_q}{\sqrt{\sum A_{lm} k_l k_m}}\right)\, dk      
  \end{split}
\end{equation}

We arrive at the following approximation result. We repeat that the
underlying observations are taken from \cite {SanSym}, our
contribution is to specify function spaces and to clarify
assumptions.
\begin{corollary}\label{cor:simplified-u}
  Let Assumption \ref {ass:a-f} be satisfied. Let $u^\eps$ be the
  solution of \eqref {eq:eps-wave} and let $v^\eps$ be defined by
  \eqref {eq:v-eps}. Then
  \begin{equation}
    \label{eq:approx-result-SanSym}
    \sup_{t\in [0,T_0 \eps^{-2}]} 
    \| u^\eps(t) - v^\eps(t) \|_{L^2(\R^n) + L^\infty(\R^n)} \le C \eps.
  \end{equation}
\end{corollary}

\begin{proof}
  The estimate for the difference $u^\eps - U^\eps$ has been concluded
  in \eqref {eq:combined-u-U}. It remains to estimate the difference
  $v^\eps - U^\eps$ in the same norm.

  With the Taylor expansion of the square root \eqref
  {eq:square-root-Taylor} we see that the definitions of $U^\eps(t)$
  and $v^\eps(t)$ coincide, except for a factor of the form
  \begin{equation*}
    \exp\left( \pm  \ri t O(\eps^4) \right) = 1 + O(\eps^2),
  \end{equation*}
  uniformly in $t$ for $t\in [0,T_0 \eps^{-2}]$.  Because of $F_0 \in
  L^\infty(\R^n)$ and the boundedness of $K$, this implies \eqref
  {eq:approx-result-SanSym}.
\end{proof}

In view of Corollary \ref {cor:simplified-u}, it will no longer be
necessary to work with $u^\eps$, the solution to the original wave
equation in a heterogeneous medium.  We can, instead, restrict
ourselves to the analysis of the function $v^\eps$, defined by \eqref
{eq:v-eps}.

Note that Taylor expansions of Bloch eigenvalues are commonly used
also in the derivation of effective equations for envelopes of
nonlinear waves in periodic structures, see e.g. \cite{DD13,DU09}.

\subsection{Symmetries}
The structure of the tensors $A$ and $C$, defined via the expansion of
$\mu_0(k)$, is very simple if we consider symmetric material functions
$a_Y$. Indeed, we will see that $A$ and $C$ are fully characterized by
three real numbers $a^*$, $\alpha$, and $\beta$.

We assume that $a_Y(.)$ is symmetric with respect to reflections
across a hyperplane $\{y_j=0\}$, $j\in \{1,\dots,n\}$, and invariant
under coordinate permutations. To be more precise, we introduce the
following transformation of $\R^n$, defined for $y = (y_1,\dots,y_n)$
as
\begin{align*}
  S_i(y) &= (y_1,\dots,y_{i-1},-y_i,y_{i+1},\dots,y_n)\,,\\
  R_{ij}(y) &= (y_1,\dots,y_{i-1},y_j,y_{i+1},\dots,y_{j-1},y_i,y_{j+1},\dots,y_n)\,.
\end{align*}
Our symmetry assumption on $a_Y$ can now be formulated as
\begin{equation}
  a_Y(y)=a_Y(S_i(y))=a_Y(R_{ij}(y))  
  \quad \text{for all } i,j\in \{1,\dots,n\} \text{ and all }  y\in\R^n\,.
\label{eq:sym_ass}
\end{equation}
As we show next, the symmetry properties of $a_Y$ in $y$ imply the
identical symmetry properties of $\mu_0$ in $k$,
\begin{align}
  \mu_0(k) & = \mu_0(S_i(k)) = \mu_0(R_{ij}(k)) \quad \text{for all }
  i,j\in \{1,\dots,n\} \text{ and all } k\in Z \label{eq:RS_sym}.
\end{align}
In fact, \eqref{eq:RS_sym} holds also for all functions $\mu_m$, but
we exploit here only the symmetry of $\mu_0$. To show
\eqref{eq:RS_sym}, we express $\mu_0(k)$ with the variational
characterization, see Theorem XIII.2 in \cite{ReedSimon4}, as
\begin{equation}
  \mu_0(k) =  \min_{\stackrel{w\in H^1_\per(Y)}{\|w\|_{L^2(Y)}=1}} I(w,k), 
  \quad \text{ where } I(w,k) := \int_Y |[(\nabla +{\rm i}k)w](y)|^2_{a_Y(y)}\, dy\,.
  \label{eq:minmax}
\end{equation}
Using the symmetry of $a_Y$, we can calculate
\begin{equation}
  \begin{split}
    I(w,S_i(k)) &= \int_Y\, 
    \left|\left[(\nabla +{\rm i}S_i(k))w\right](y)\right|^2_{ a_Y(y)}\, dy \\
    &= \int_{S^{-1}_i(Y)} \,
    \left|\left[(\nabla +{\rm i}S_i(k))w\right] (S_i(\tilde{y}))\right|^2_{a_Y(\tilde{y})}\, d\tilde{y}\\
    &= \int_{Y}\,
    \left| S_i \left(\left[(\nabla +{\rm i}k) (w\circ S_i)\right]
        (\tilde{y})\right)\right|_{a_Y(\tilde{y})}^2\, d\tilde{y} 
    = I(w\circ S_i,k)\,.
  \end{split}
  \label{eq:sym_pf}
\end{equation}
Minimizing over the functions $w\circ S_i$ provides the same result as
minimizing over $w$, since with $w\in H^1_\per(Y)$ also $w\circ S_i
\in H^1_\per(Y)$. This provides \eqref{eq:RS_sym} for $S_i$. The
calculation for $R_{ij}$ is identical.

As a consequence of the symmetry, we obtain the following
characterization of the Taylor expansion coefficients $A$ and $C$.
\begin{lemma}
  \label{lem:sym}
  Let $a_Y$ have the symmetries \eqref{eq:sym_ass}. Then the tensors
  $A$ and $C$, defined in \eqref {eq:Taylor-mu}, satisfy
  \begin{align*}
    &A_{ii}=A_{11} =: a^*, \qquad\quad A_{ij} = 0,\\
    &C_{iiii} =C_{1111}=:\alpha, 
    \qquad\  C_{ijij}=C_{ijji}=C_{iijj}=C_{1122}=:\beta
  \end{align*}
  for all $i,j\in \{1,\dots,n\}$ with $i\neq j$. All entries of $C$,
  that are not mentioned above, vanish.
\end{lemma}

\begin{proof}
  The proof uses the symmetry \eqref{eq:RS_sym}. The symmetry under
  $S_i$ implies that $\mu_0$ is an even function. Thus all derivatives
  of $\mu_0$ with an odd number of derivatives in one variable vanish
  at $k=0$. This proves $A_{ij} =0$ and, e.g., $C_{iiij} = 0$.  The
  fact that derivatives can be interchanged provides, e.g., $C_{iijj}
  = C_{ijij}$.

  The symmetry under $R_{ij}$ allows to calculate
  \begin{align*}
    \partial_{k_i}^2\mu_0 (k)
    = \partial_{k_i}^2 (\mu_0\circ R_{ij})(k)
    = [\partial_{k_j}^2 \mu_0](R_{ij}(k))\,.
  \end{align*}
  Evaluating in $k=0$ provides $A_{ii} = A_{jj}$. The analogous
  calculation for fourth order derivatives shows, e.g., $C_{iiii} =
  C_{jjjj}$. This proves the claim in the two-dimensional case.

  For $n\ge 3$ we can analogously use the symmetry under $R_{jl}$ to
  get $C_{iijj}=C_{iill}$ for all indices $1\leq i,j,l\le n$ with $i,
  j, l$ distinct.
\end{proof}

\begin{remark}
  \label{rem:odd}
  Independent of spatial symmetry assumptions on $a_Y$, odd
  derivatives of $\mu_0$ vanish in $k=0$.

  Let us sketch the proof for this fact: Due to the equivalence of the
  reflection $k\leftrightarrow -k$ and the complex conjugation in
  \begin{equation*}
    I(w,-k) = \int_Y |(\nabla-{\rm i}k)w|_{a_Y}^2
    = \int_Y |(\nabla+{\rm i}k) w^*|_{a_Y}^2 = I(w^*,k)
  \end{equation*}
  and the fact $w\in H^1_{\text{per}}(Y)\Leftrightarrow w^*\in
  H^1_{\text{per}}(Y)$, we get
  \begin{equation*}
    \mu_0(k)=\mu_0(-k)\quad\text{for all }k\in Z.
  \end{equation*}
 
  As in the proof of Lemma \ref{lem:sym} one obtains
  $\del_{k_i}\del_{k_j}\del_{k_l}\mu_0(k) =
  -\del_{k_i}\del_{k_j}\del_{k_l}\mu_0(-k)$ for all
  $i,j,l\in\{1,...,n\}$ and all $k\in Z$, and hence
  $\del_{k_i}\del_{k_j}\del_{k_l}\mu_0(0)=0$.  The argument can be
  used for arbitrary odd derivatives.
\end{remark}

\section{A well-posed weakly dispersive equation} 
\label{sec.wellposed}

A weakly dispersive equation that is related to the definition of
$v^\eps$ is (at this point, we correct a typo of \cite{SanSym}
regarding the sign before $C$)
\begin{equation}
 \label{eq:formal-weakly-disp}
  \del_t^2 u = A D^2 u - \eps^2 C D^4 u\,.
\end{equation}
Indeed, when applied to $v^\eps$ of \eqref {eq:v-eps}, the operator
$AD^2$ produces the factor $-A_{lm} k_l k_m$ under the integral, and
the operator $-\eps^2 CD^4$ produces the factor $-\eps^2 C_{lmnq} k_l
k_m k_n k_q$.  The second time derivative produces the
factor $$-A_{lm} k_l k_m - \eps^2 C_{lmnq} k_l k_m k_n k_q - (\eps^4
/4) (C_{lmnq} k_l k_m k_n k_q)^2 / (A_{lm} k_l k_m)$$ under the
integral. Therefore, up to an error of order $\eps^4$, the function
$v^\eps$ solves \eqref {eq:formal-weakly-disp}.

We emphasize that, in general, \eqref {eq:formal-weakly-disp} cannot
be used as an effective dispersive model. The fourth order operator
$-CD^4$ on the right hand side can be positive such that \eqref
{eq:formal-weakly-disp} is ill-posed. In the one-dimensional setting,
$C < 0$ is shown in \cite {Lamacz-Disp} (compare also
\cite{MR2219790}), hence the equation is necessarily ill-posed.
Section \ref{S:num-2D} includes a two-dimensional numerical example
where the numbers $\alpha$ and $\beta$, describing $C$, satisfy
$\alpha<0$ and $\beta>0$. Moreover, there holds $3\beta < |\alpha|$,
such that $-CD^4$ is a positive operator.

As a consequence, even though $v^\eps$ solves \eqref
{eq:formal-weakly-disp} up to an error of order $\eps^4$, we cannot
conclude that solutions to this equation provide approximations of
$v^\eps$. Even worse, it may be impossible to construct any solution
of \eqref {eq:formal-weakly-disp}.

\subsection{Decomposition of the operator for symmetric media}

As indicated in the introduction, our aim is now to replace \eqref
{eq:formal-weakly-disp} by a well-posed equation, which is equivalent
in all relevant powers of $\eps$. We therefore start from the two
tensors $A = a^*\, \id\in \R^{n\times n}$ and $C\in \R^{n\times
  n\times n\times n}$ of Lemma \ref {lem:sym} and consider the operator
\begin{equation}
  \label{eq:operator-C}
  C D^4 = \sum_{ijkl} C_{ijkl} \del_i \del_j \del_k \del_l
  =  \alpha \sum_{i=1}^n\del^4_i 
  + 3\beta \sum_{\stackrel{i,j=1}{i\neq j}}^n \del^2_i\del^2_j.
\end{equation}
To avoid confusion, we note that $\sum_{i\neq j} = 2\sum_{i<j}$.
Our aim is to construct coefficients $E \in \R^{n\times n}$ and $F \in
\R^{n\times n\times n\times n}$ such that the differential operator
can be re-written as
\begin{equation}
  \label {eq:def-coeff-EF}
  -C D^4 = E D^2 A D^2- F D^4\,,
\end{equation}
where $E$ and $F$ are positive semidefinite and symmetric, i.e.
\begin{align}\label{eq:admissibledecompose}
  \sum_{i,j,k,l=1}^n F_{ijkl}\xi_{ij}\xi_{kl}\geq 0\quad 
  \text{ for  every }\xi\in\R^{n\times n}
  \quad\text{ and }\,    F_{ijkl}=F_{klij}
\end{align}
and $\sum_{i,j=1}^n E_{ij}\eta_i\eta_j \geq 0$ for every $\eta\in\R^n$
and $E_{ij}=E_{ji}$ for $i,j,k,l \in\{1,...,n\}$.  The decomposition
result \eqref {eq:def-coeff-EF} allows, using the lowest order of
\eqref {eq:formal-weakly-disp}, to re-write the operator in the
evolution equation formally as
\begin{equation}
  \label{eq:replacement}
  - \eps^2 C D^4 u =  \eps^2 E D^2 A D^2 u -  \eps^2 F D^4 u
  = \eps^2 E D^2 \del_t^2 u - \eps^2 F D^4 u + O(\eps^4)\,.
\end{equation}
With this replacement in equation \eqref {eq:formal-weakly-disp}, we
obtain the well-posed equation \eqref {eq:weakly-dispersive}.

\begin{lemma}[Decomposability]\label{lem:decompose}
  Let $A\in \R^{n\times n}$ and $C\in \R^{n\times n\times n\times n}$
  be as in Lemma \ref {lem:sym}, given by three constants $a^*> 0$,
  $\alpha, \beta\in \R$, in particular with $CD^4$ given by \eqref
  {eq:operator-C}.  Then there exist symmetric and positive
  semidefinite tensors $E\in \R^{n\times n}$ and $F\in \R^{n\times
    n\times n\times n}$ such that $CD^4$ can be written as in
  \eqref{eq:def-coeff-EF}.

  Using $\{ a \}_+ := \max\{a,0\}$ to denote the positive part of a
  number $a$, a possible choice of $E$ and $F$ is
  \begin{align}
    \label{eq:decompose1}
    E_{ii}&=\frac1{a^*}\left(\{-\alpha\}_++3\{-\beta\}_+\right),
    \qquad E_{ij}=0,\\
    \label{eq:decompose2}
    F_{iiii}&=\{\alpha\}_++3\{-\beta\}_+,\hskip0.5cm
    F_{ijij}=\{-\alpha\}_++3\{\beta\}_+,
  \end{align}
  for all $i,j\in\{1,...,n\}$ with $i\neq j$.  All other entries of
  $F$ are set to zero.
\end{lemma}

With \eqref{eq:decompose1}--\eqref{eq:decompose2}, we introduce the
two differential operators
\begin{align*}
  E D^2&= \frac1{a^*}\left(\{-\alpha\}_++3\{-\beta\}_+\right)\sum_{i=1}^n\del^2_i
  =\frac1{a^*}\left(\{-\alpha\}_++3\{-\beta\}_+\right)\Delta,\\
  F D^4&= \left(\{\alpha\}_++3\{-\beta\}_+\right)\sum_{i=1}^n\del^4_i
  + \left(\{-\alpha\}_++3\{\beta\}_+\right)\sum_{i,j=1, i\neq
    j}^n\del^2_i\del^2_j.
\end{align*}

Since $\alpha$ and $\beta$ are real numbers, there are four different
possibilities for the signs of $\alpha$ and $\beta$. Distinguishing
these four cases, we can write the two differential operators in very
simple expressions.
 
\begin{remark} \label{r:cases}The operators $E D^2$ and $F D^4$ of
  \eqref{eq:decompose1}--\eqref{eq:decompose2} are given as follows.

  \begin{description}
  \item[Case 1.] $\alpha\leq0,\beta\leq 0$:
   \begin{equation*}
     ED^2=\frac1{a^*}(|\alpha|+3|\beta|)\Delta
     \hspace{0.3cm}\text{ and }\hspace{0.3cm}
     F D^4=3|\beta|\sum_{i=1}^n\del^4_i+|\alpha|
     \sum_{i,j=1, i\neq j}^n\del^2_i\del^2_j
   \end{equation*}
 \item[Case 2.]  $\alpha\leq0, \beta > 0$:
   \begin{equation*}
     ED^2=\frac{|\alpha|}{a^*}\Delta\hspace{0.3cm}\text{ and }\hspace{0.3cm}
     F D^4=(|\alpha|+3\beta)
     \sum_{i,j=1, i\neq j}^n\del^2_i\del^2_j.
   \end{equation*}
 \item[Case 3.] $\alpha > 0, \beta\leq 0$:
    \begin{equation*}
     ED^2=\frac{3|\beta|}{a^*}\Delta\hspace{0.3cm}\text{ and }\hspace{0.3cm}
     F D^4=(\alpha +3|\beta|)\sum_{i=1}^n\del^4_i
   \end{equation*}
 \item[Case 4.] $\alpha\geq0,\beta\geq 0$:  
   \begin{equation*}
     ED^2=0\hspace{0.3cm}\text{ and }\hspace{0.3cm}
     F D^4=\alpha\sum_{i=1}^n\del^4_i+3\beta
     \sum_{i,j=1, i\neq j}^n\del^2_i\del^2_j=CD^4.
   \end{equation*}
  \end{description} 
  We note that the first two cases (with $\alpha\leq 0$) are the
  relevant ones in our numerical examples.
\end{remark}

\begin{proof}[Proof of Lemma \ref{lem:decompose}] {\em Step
    1. Properties of $E$ and $F$.} By definition, $E$ is a nonnegative
  multiple of the identity in $\R^n$. The tensor is therefore positive
  semidefinite and symmetric.  Also $F$ is symmetric by definition.
  For $\xi\in\R^{n\times n}$ there holds
  \begin{align*}
    &\sum_{i,j,k,l=1}^n F_{ijkl}\xi_{ij}\xi_{kl}\\
    &\quad =\sum_{i=1}^n\left(\{\alpha\}_++3\{-\beta\}_+\right)(\xi_{ii})^2
    +\sum_{i,j=1, i\neq j}^n
    \left(\{-\alpha\}_++3\{\beta\}_+\right)(\xi_{ij})^2\geq 0.
  \end{align*}
  Hence $F$ is also positive semidefinite.

  {\em Step 2. Decomposition property.}  It remains to show $-C D^4 =
  E D^2 A D^2- F D^4$. For that purpose we calculate the right hand
  side as
  \begin{align*}
    &E D^2 A D^2 - F D^4\\
    &\quad=\frac1{a^*}\left(\{-\alpha\}_+
      +3\{-\beta\}_+\right)\sum_{i=1}^n\del^2_i\left(\sum_{j=1}^n
      a^* \del^2_j\right)
    -\left(\{\alpha\}_++3\{-\beta\}_+\right)\sum_{i=1}^n\del^4_i 
    \displaybreak[2]\\
    &\qquad - \left(\{-\alpha\}_++3\{\beta\}_+\right)\sum_{i,j=1, i\neq
      j}^n\del^2_i\del^2_j \displaybreak[2]\\
    &\quad = \left(\{-\alpha\}_++3\{-\beta\}_+\right)\sum_{i=1}^n\del^4_i +
    \left(\{-\alpha\}_++3\{-\beta\}_+\right)\sum_{i,j=1, i\neq
      j}^n\del^2_i\del^2_j\\
    &\qquad -\left(\{\alpha\}_++3\{-\beta\}_+\right)\sum_{i=1}^n\del^4_i -
    \left(\{-\alpha\}_++3\{\beta\}_+\right)\sum_{i,j=1, i\neq
      j}^n\del^2_i\del^2_j \displaybreak[2]\\
    &\quad= -\alpha\sum_{i=1}^n\del^4_i-3\beta\sum_{i,j=1, i\neq
      j}^n\del^2_i\del^2_j =-C D^4\,.
  \end{align*}
  This is the desired decomposition \eqref{eq:def-coeff-EF}.
\end{proof}

\subsection{An approximation result}

With the subsequent theorem, we provide the central error estimate for
our main result. We start from two tensors $A$ and $C$ (in the
application of the theorem they are defined by
\eqref{eq:Taylor-mu}), and assume that $C$ is decomposable with
tensors $E$ and $F$. With these four tensors we can study two objects:
The solution $w^\eps$ of \eqref {eq:weakly-dispersive}, and the
function $v^\eps$, defined by the representation formula \eqref
{eq:v-eps}. Our next theorem compares these two objects.

\begin{theorem}\label{thm:main-approximation}
  Let $A,C,E,F$ be tensors with the properties: $A$ is symmetric and
  positive definite, $\sum_{ij}A_{ij}\xi_i\xi_j\geq \gamma |\xi|^2$
  for some $\gamma>0$, $E$ and $F$ are positive semidefinite and
  symmetric, $C$ allows the decomposition \eqref {eq:def-coeff-EF}.
  Then the following holds.
  \begin{enumerate}
  \item {\bf Well-posedness.} Let $R\in L^1(0,T_0\eps^{-2} ;
    L^2(\R^n))$ be a right hand side and let $f\in H^2(\R^n)$ be an
    initial datum. We study an inhomogeneous version of equation
    \eqref {eq:weakly-dispersive},
    \begin{equation}\label{eq:weakly-dispersive-copy}
      \begin{split}
        &\del_t^2 w^\eps(x,t)
        - A D^2 w^\eps(x,t) - \eps^2 \del_t^2 E D^2 w^\eps(x,t) 
        + \eps^2 F D^4 w^\eps (x,t)=R(x,t)\,,\\
        &w^\eps(x,0) = f(x), \qquad \del_t w^\eps(x,0) = 0.
      \end{split}      
    \end{equation}
    for $x\in\R^n$ and $t\in (0,T_0\eps^{-2})$.  This equation has a
    unique solution $w^\eps \in L^\infty(0,T_0\eps^{-2} ; H^2(\R^n))
    \cap W^{1,^\infty}(0,T_0\eps^{-2} ; H^1(\R^n))$.\smallskip
  \item {\bf Approximation.} Let $v^\eps$ be defined by \eqref
    {eq:v-eps} with $F_0$ and $f$ related by \eqref {eq:f-cond}. Let
    $w^\eps$ be a solution of \eqref{eq:weakly-dispersive-copy} to
    $R\equiv 0$. Then
    \begin{equation}
      \label{eq:approx-result-v}
      \sup_{t\in [0,T_0 \eps^{-2}]}\|\del_t( v^\eps - w^\eps)(.,t) \|_{L^2(\R^n)}
      + \sup_{t\in [0,T_0 \eps^{-2}]}  \|\nabla( v^\eps - w^\eps)(.,t)\|_{L^2(\R^n)}
      \leq  C\eps^2\,,
    \end{equation}
    where $C>0$ denotes a constant that depends on $f$ and the
    coefficients, but is independent of $\eps$.
  \end{enumerate}
\end{theorem}

\begin{proof} {\it Well-posedness of problem
    \eqref{eq:weakly-dispersive-copy}.}
  We use the following concept of weak solutions. We say that
  $w^\eps\in L^\infty(0,T_0\eps^{-2};H^2(\R^n))$ with the property
  $\del_t w^\eps\in L^\infty(0,T_0\eps^{-2};H^1(\R^n))$ is a weak
  solution, if it satisfies $w^\eps(x,0)=f(x)$ in the sense of traces
  and if
  \begin{align}
    \begin{split}
      \label{eq:defweaksol}
      \int_0^{T_0 \eps^{-2}}\int_{\R^n} R\,\phi=&
      \int_0^{T_0 \eps^{-2}}\int_{\R^n}\{-\del_t w^\eps\del_t\phi 
      + \nabla\phi\cdot A\nabla w^\eps\}\\
      &+\eps^2\int_0^{T_0 \eps^{-2}}\int_{\R^n}\{-\nabla(\del_t\phi)
      \cdot E\nabla(\del_t w^\eps)+ D^2\phi: FD^2w^\eps\}
    \end{split}
  \end{align}
  for every test-function $\phi\in
  C_c^1([0,T_0\eps^{-2});H^2(\R^n))$. Here $D^2\phi : F D^2w^\eps$
  denotes the tensor product of $D^2\phi$ and $F D^2w^\eps$,
  $$D^2\phi : F D^2w^\eps
  := \sum_{i,j,k,l=1}^n\del_i\del_j\phi F_{ijkl}\del_k\del_lw^\eps. $$

  We prove the existence of a weak solution to problem
  \eqref{eq:weakly-dispersive-copy} with a Galerkin scheme.  We use a
  countable basis $\{\psi^k\}_{k\in\N}$ of the separable space
  $H^1(\R^n)$ and the finite-dimensional sub-spaces
  $V_K:=span\{\psi^1,...,\psi^K\}\subset H^1(\R^n)$.  The basis
  $\{\psi^k\}_{k\in\N}$ is chosen in such a way that the functions
  $\psi^k$ are of class $H^2(\R^n)$ and such that the family of
  $L^2$-orthogonal projections $P_K$ onto $V_K$ are bounded as maps
  $P_K : H^2(\R^n) \to H^2(\R^n)$. For every $K\in\N$ we search for
  approximative solutions $w^\eps_K$ of the form
  \begin{equation*}
    w^\eps_K:[0,T_0\eps^{-2}]\rightarrow V_K,\quad 
    w^\eps_K(t)= \sum_{k=1}^K b^\eps_k(t) \psi^k
  \end{equation*}
  with coefficients $b^\eps_k:[0,T_0\eps^{-2}]\rightarrow \R$.  We
  demand that $w^\eps_K$ solves \eqref{eq:weakly-dispersive-copy} in
  the weak sense, however, only for test-functions in the
  $K$-dimensional space $V_K$,
  \begin{align}
    \begin{split}
      \label{eq:weakformprojected}
      \int_{\R^n} R\psi^k =&
      \int_{\R^n}\{\del^2_t w^\eps_K\, \psi^k + \nabla\psi^k\cdot A\nabla w^\eps_K\}\\ 
      &+ \eps^2\int_{\R^n}\{\nabla\psi^k\cdot E\nabla(\del^2_t w^\eps_K) 
      + D^2\psi^k:F D^2w^\eps_K\}
    \end{split}
  \end{align}
  for every $k\in\{1,...,K\}$. For the initial data we demand that
  $\la w^\eps_K|_{t=0}, \psi^k\ra_{L^2(\R^n)}=\la f,
  \psi^k\ra_{L^2(\R^n)}$ and $\la \del_t w^\eps_K|_{t=0},
  \psi^k\ra_{L^2(\R^n)}=0$.  For every $K\in \N$, equation
  \eqref{eq:weakformprojected} is a $K$-dimensional system of ordinary
  differential equations of second order for the coefficient vector
  $(b^\eps_1(t), \hdots ,b^\eps_K(t))$, which can be solved
  uniquely. This provides the approximative solutions $w^\eps_K$.

  We now derive $K$-independent a priori estimates for the sequence
  $w^\eps_K$.  For that purpose we test equation
  \eqref{eq:weakly-dispersive-copy} with $\del_t w^\eps_K$ (more precisely,
  we multiply \eqref{eq:weakformprojected} by $\del_t b^\eps_k$ and
  take the sum over $k$). Exploiting the symmetry of $A,E$ and $F$ we
  obtain
  \begin{align}
    \begin{split}
      \label{eq:test1}
      \int_{\R^n} R\,\del_t w^\eps_K=&\frac12\del_t\int_{\R^n}\{|\del_t w^\eps_K|^2 
      + \nabla w^\eps_K\cdot A\nabla w^\eps_K\}\\
      &+\eps^2\frac12\del_t\int_{\R^n}\{\nabla(\del_t w^\eps_K)
      \cdot E\nabla(\del_t w^\eps_K) 
      + D^2 w^\eps_K:F D^2 w^\eps_K\}.
    \end{split}
  \end{align}
  We next integrate \eqref{eq:test1} over $[0,t_0]$, where
  $t_0\in[0,T_0\eps^{-2}]$ is arbitrary. We exploit the initial
  condition $w^\eps_K|_{t=0} = f_K$, where $f_K$ is the
  $L^2$-projection of $f$ onto $V_K$. The other initial condition is
  $\del_t w^\eps_K|_{t=0}=0$ and we arrive at
  \begin{align}
    &2\int_0^{t_0}\int_{\R^n}R\,\del_tw^\eps_K  
    + \int_{\R^n}\nabla f_K\cdot A\nabla f_K 
    + \eps^2\int_{\R^n} D^2f_K:FD^2f_K\nonumber\\
    &\quad= \int_{\R^n}\{|\del_t w^\eps_K|_{t=t_0}|^2 
    + \nabla w^\eps_K|_{t=t_0}\cdot A\nabla w^\eps_K|_{t=t_0}\}\nonumber\\
    &\qquad 
    +\eps^2\int_{\R^n}\{\nabla(\del_t w^\eps_K)|_{t=t_0}
    \cdot E\nabla(\del_t w^\eps_K)|_{t=t_0}
    +D^2w^\eps_K|_{t=t_0}:FD^2w^\eps_K|_{t=t_0}\}\nonumber\\
    \label{eq:test2}
    &\quad\geq\|\del_t w^\eps_K(.,t_0)\|^2_{L^2(\R^n)} 
    +\gamma\|\nabla w^\eps_K(.,t_0)\|^2_{L^2(\R^n)}\,.
  \end{align}
  In the last line we exploited that $A$ is positive definite with
  parameter $\gamma>0$ and that $E$ and $F$ are positive
  semi-definite. Introducing $Y(t) := \| \del_t w^\eps_K(.,t)
  \|^2_{L^2(\R^n)} + \gamma\|\nabla w^\eps_K(.,t)\|^2_{L^2(\R^n)}$
  for the right hand side of \eqref {eq:test2} and $Y_0 :=
  \int_{\R^n}\{\nabla f_K\cdot A\nabla f_K + \eps^2 D^2f_K:FD^2f_K\}$,
  we can calculate with the Cauchy-Schwarz inequality 
  \begin{equation}\label{eq:gronwalltype1}
    \begin{split}
      Y(t)&\leq 2\int_0^t \|R(., s)\|_{L^2(\R^n)} 
      \|\del_t w^\eps_K(.,s)\|_{L^2(\R^n)}\,ds   + Y_0\\
      &\leq 2\int_0^t\|R(., s)\|_{L^2(\R^n)}\sqrt{Y(s)}\,ds + Y_0 \,.      
    \end{split}
  \end{equation}
  We claim that a Gronwall-type argument leads from inequality
  \eqref{eq:gronwalltype1} to the estimate
  \begin{align}
    \label{eq:gronwalltype2}
    Y(t)\leq 2 Y_0
    + 2\left(\int_0^t\|R(.,s)\|_{L^2(\R^n)}\,ds\right)^2\,,
  \end{align}
  see Appendix \ref {app.Gronwall}.  With inequality
  \eqref{eq:gronwalltype2} at hand we finally obtain the following a
  priori estimate
  \begin{align}
    \begin{split}
      \label{eq:apriori1}
      \sup_{t\in[0,T_0\eps^{-2}]}
      Y(t)&=\sup_{t\in[0,T_0\eps^{-2}]} \left\{ \|\del_t w^\eps_K(.,t)\|^2_{L^2(\R^n)}
      + \gamma \|\nabla w^\eps_K(.,t)\|^2_{L^2(\R^n)} \right\}\\
      &\leq 2 Y_0 +  2\|R\|^2_{L^1(0,T_0\eps^{-2};L^2(\R^n))}\\
      &\leq 2(C(A)+\eps^2 C(F))\|f\|^2_{H^2(\R^n)} +
      2\|R\|^2_{L^1(0,T_0\eps^{-2};L^2(\R^n))}.
    \end{split}
  \end{align} 
  The bound in \eqref{eq:apriori1} is independent of $K$. Hence,
  possibly after passing to a subsequnce, we may consider the weak
  limit $K\rightarrow\infty$ of solutions $w^\eps_K$ of the Galerkin
  scheme. Due to the linearity of the problem, the limit provides a
  solution $w^\eps\in L^\infty(0,T_0\eps^{-2};H^1(\R^n))$ with $\del_t
  w^\eps\in L^\infty(0,T_0\eps^{-2};L^2(\R^n))$ to
  \eqref{eq:weakly-dispersive-copy} in the sense of distributions.
  Furthermore, $w^\eps$ satisfies exactly the same a priori estimates
  as its approximations $w^\eps_K$.  By differentiating \eqref
  {eq:weakly-dispersive-copy} with respect to $x$, one discovers that
  $w^\eps$ has in fact higher spatial regularity and that the
  distributional solution $w^\eps$ is in fact a weak solution in the
  sense of \eqref{eq:defweaksol}.  Note that the uniqueness of
  solutions to problem \eqref {eq:weakly-dispersive-copy} is a direct
  consequence of the a priori estimate \eqref{eq:apriori1}.  Hence,
  the weakly dispersive problem is well-posed.

  \medskip {\it Proof of the approximation result
    \eqref{eq:approx-result-v}.} By applying the differential operator
  $\del_t^2- A D^2 -\eps^2 \del^2_t ED^2+ \eps^2 F D^4$ to $v^\eps$,
  which is explicitly given in \eqref{eq:v-eps}, one immediately
  discovers that $v^\eps$ solves Equation
  \eqref{eq:weakly-dispersive-copy} with a right hand side of order
  $\eps^4$. More precisely, we calculate first with the decomposition
  of the operator $-C D^4 = E D^2 A D^2- F D^4$
  \begin{equation*}
    \del_t^2v^\eps - A D^2v^\eps 
    =  - \eps^2 C D^4 v^\eps + \eps^4 \tilde R^\eps
    = \eps^2 E D^2 A D^2 v^\eps - \eps^2 F D^4 v^\eps + \eps^4 \tilde R^\eps\,,
  \end{equation*}
  where the error term comes from the double differentiation of the
  last factor of $v^\eps$ with respect to time, 
  \begin{align*}
    \tilde R^\eps &:= - \frac{1}{8} (2\pi)^{-n/2} \sum_{\pm} \int_{k\in K} 
    \frac{(\sum C_{lmnq}k_lk_mk_nk_q)^2}{\sum A_{lm} k_lk_m}
    F_0(k)\\
    &\times\ \exp\left( \ri k\cdot x \pm \ri \sqrt{\sum A_{lm} k_l k_m} t \right)
       \exp\left( \pm \frac{\ri \eps^2}{2} t
      \frac{\sum C_{lmnq} k_l k_m k_n k_q}{\sqrt{\sum A_{lm} k_l k_m}}\right)\ dk.
  \end{align*}
  With this preparation we can now evaluate the application of the
  full differential operator as
  \begin{align}
    \begin{split}
      \label{eq:weakdispperturb1}
      &\del_t^2 v^\eps - A D^2v^\eps 
      + \eps^2 F D^4v^\eps - \eps^2 \del^2_t ED^2v^\eps\\
      &\qquad= \eps^2 E D^2 A D^2 v^\eps + \eps^4 \tilde R^\eps
      -\eps^2 \del^2_t ED^2v^\eps\\
      &\qquad= \eps^2 E D^2 (A D^2 v^\eps - \del^2_t  v^\eps)
      + \eps^4 \tilde R^\eps\\
      &\qquad= \eps^4 E D^2 (C D^4 v^\eps - \eps^2 \tilde R^\eps)
      + \eps^4 \tilde R^\eps
      =: R^\eps\,.
    \end{split}
  \end{align}
  In particular, $\sup_{t\in [0,T_0 \eps^{-2}]}
  \|R^\eps(.,t)\|_{L^2(\R^n)} \leq \tilde C\eps^4$ for some
  $\eps$-independent constant $\tilde C$.  Due to the linearity of the
  problem and the fact that $w^\eps$ is a solution to
  \eqref{eq:weakly-dispersive-copy} with $R\equiv 0$, the difference
  $v^\eps-w^\eps$ solves equation \eqref{eq:weakdispperturb1}
  \begin{align*}
    \del_t^2&(v^\eps-w^\eps)- A D^2(v^\eps-w^\eps) 
    + \eps^2 F D^4(v^\eps-w^\eps)-\eps^2 \del^2_t ED^2(v^\eps-w^\eps) =R^\eps,
  \end{align*}
  with vanishing initial data $(v^\eps-w^\eps)(.,0) =
  \del_t(v^\eps-w^\eps)(.,0) = 0$.

  By applying the a priori estimate \eqref{eq:apriori1} to the
  difference $(v^\eps-w^\eps)$ we obtain
  \begin{align*}
    &\sup_{t\in[0,T_0\eps^{-2}]} \left\{ \|\del_t (v^\eps-w^\eps)(.,t)\|^2_{L^2(\R^n)} 
    + \gamma \|\nabla (v^\eps-w^\eps)(.,t)\|^2_{L^2(\R^n)} \right\}\\
    &\qquad \leq 2\|R^\eps\|^2_{L^1(0,T_0\eps^{-2};L^2(\R^n))}
    \leq 2(T_0\eps^{-2}\|R^\eps\|_{L^\infty(0,T_0\eps^{-2};L^2(\R^n))})^2
    \leq  C\eps^4,
  \end{align*} 
  where in the last step we exploited that
  $\|R^\eps\|_{L^\infty(0,T_0\eps^{-2};L^2(\R^n))}\leq C\eps^4$.  This
  implies \eqref{eq:approx-result-v}.
\end{proof}

\paragraph{The main theorem.} Theorem \ref {thm:main} is a consequence
of the previous results.

\begin{proof} We have seen in Lemma \ref {lem:sol-expansion}, that the
  solution $u^\eps$ permits the expansion \eqref
  {eq:Bloch-complete-soln-eps} in Bloch-waves.  In Theorem \ref
  {thm:SanSym-1} we have seen that only the term $m=0$ has to be
  considered.

  We concluded with \eqref {eq:approx-result-SanSym} a smallness
  result, that $\| u^\eps - v^\eps \|_{L^2+L^\infty}$ is of order
  $\eps$. The norms coincide with the ones in the claimed result
  \eqref {eq:approx-result} for $\| u^\eps - w^\eps \|$.

  Finally, Theorem \ref {thm:main-approximation} provides the
  well-posedness claim and the estimate \eqref {eq:approx-result-v},
  which shows that norms of derivatives of $v^\eps - w^\eps$ are of
  order $\eps^2$. The subsequent Lemma \ref {lem:interpolation}
  provides the estimate for $\| v^\eps - w^\eps \|_{L^2+L^\infty}$ of
  order $\eps$, i.e. in the norm of \eqref {eq:approx-result}.
\end{proof}

\begin{lemma}\label{lem:interpolation}
  For $n\ge 1$ and $T>0$ fixed, let $g^\eps: \R^n\times [0,T/\eps^2]
  \to \R$ be a sequence of functions with $g^\eps(.,0) \equiv
  0$. Then, with an $\eps$-independent constant $C>0$, there holds
  \begin{equation}
    \label{eq:interpolation}
    \begin{split}
      &\sup_{t\in [0,T/\eps^2]} \| g^\eps(.,t) \|_{L^2(\R^n) + L^\infty(\R^n)} \\
      &\qquad \le C \eps^{-1} \sup_{t\in [0,T/\eps^2]}  
      \left\{ \| \del_t g^\eps(.,t) \|_{L^2(\R^n)} 
        + \| \nabla g^\eps(.,t) \|_{L^2(\R^n)} \right\}\,.      
    \end{split}
  \end{equation}
\end{lemma}

\begin{proof}
  We first consider $n\ge 2$.  Given $\eps>0$, we choose a tiling of
  the space as
  \begin{equation}
    \label{eq:tiling-1}
    \R^n = \bigcup_{m\in \Z^n} E_m^\eps\,,\qquad
    E_m^\eps = x_m + [0,\eps^{-1})^n\,,\qquad
    x_m = m \eps^{-1}\,.
  \end{equation}
  Given the function $g^\eps$ we define a piecewise constant function
  through an averaging procedure,
  \begin{equation}
    \label{eq:bargeps}
    \bar g^\eps(x,t) := \meanint_{E_m^\eps} g^\eps(\xi,t)\, d\xi\quad
    \text{ if } x\in E_m^\eps\,.
  \end{equation}
  The Poincar\'e inequality for functions with vanishing average
  allows to estimate
  \begin{align*}
     &\| g^\eps(.,t) -  \bar g^\eps(.,t) \|_{L^2(\R^n)}^2
     = \sum_m \| g^\eps(.,t) -  \bar g^\eps(.,t) \|_{L^2(E^\eps_m)}^2\\
     &\qquad \le C\, \diam(E_m^\eps)^2 \sum_m \| \nabla g^\eps(.,t) \|_{L^2(E^\eps_m)}^2
     \le C \eps^{-2} \| \nabla g^\eps(.,t) \|_{L^2(\R^n)}^2 \,.
  \end{align*}
  This provides estimate \eqref {eq:interpolation} for the part
  $g^\eps - \bar g^\eps$.

  In order to estimate $\bar g^\eps$, we use the fact that averaging
  does not increase the $L^2$-norm,
  \begin{align*}
    \sum_m |E_m^\eps| |\del_t \bar g^\eps(x_m,t)|^2
    = \| \del_t \bar g^\eps(.,t)\|_{L^2(\R^n)}^2
    \le \| \del_t g^\eps(.,t)\|_{L^2(\R^n)}^2\,.
  \end{align*}
  With the fundamental theorem of calculus we find
  \begin{align*}
    \| \bar g^\eps(.,t)\|_{L^\infty(\R^n)}^2 &= \max_m  |\bar g^\eps(x_m,t)|^2
    \le \max_m  \frac{T^2}{\eps^4} \sup_{s\in [0,T\eps^{-2}]} 
    |\del_t \bar g^\eps(x_m,s)|^2\\
    &\le  \frac{T^2}{\eps^4} |E_m^\eps|^{-1} 
    \sup_{s\in [0,T\eps^{-2}]} \sum_m |E_m^\eps| |\del_t \bar g^\eps(x_m,s)|^2\\
    &\le T^2 \eps^{n-4} \sup_{s\in [0,T\eps^{-2}]} \| \del_t g^\eps(.,s)\|_{L^2(\R^n)}^2\,.
  \end{align*}
  For $n\ge 2$, this provides estimate \eqref {eq:interpolation} for
  the remaining part $\bar g^\eps$.  

  \smallskip In the case $n=1$ we proceed in a similar way, using now
  a tiling with pieces of larger diameter,
  \begin{equation}
    \label{eq:tiling-2}
    \R = \bigcup_{m\in \Z} E_m^\eps\,,\qquad
    E_m^\eps = x_m + [0,\eps^{-2})\,,\qquad
    x_m = m \eps^{-2}\,.
  \end{equation}
  The estimate for $\bar g^\eps\in L^\infty(0,T\eps^{-2};
  L^\infty(\R^n))$ is obtained as above with the $\eps$-factor
  $\eps^{-4} |E_m^\eps|^{-1} = \eps^{-2}$ as desired. To estimate the
  difference $g^\eps - \bar g^\eps$ we use, in the case $n=1$, the
  same $L^\infty$-based norm. We calculate, for arbitrary $t\in
  (0,T\eps^{-2})$,
  \begin{align*}
    &\| g^\eps(.,t) - \bar g^\eps(.,t) \|_{L^\infty(\R^n)}
    = \sup_m \| g^\eps(.,t) -  \bar g^\eps(.,t) \|_{L^\infty(E^\eps_m)}\\
    &\qquad \le \sup_m \| \del_x g^\eps(.,t) \|_{L^1(E^\eps_m)} \le
    \sup_m\ \diam(E_m^\eps)^{1/2}\ \| \del_x g^\eps(.,t) \|_{L^2(E^\eps_m)}\,.
  \end{align*}
  Because of $\diam(E_m^\eps)^{1/2} = \eps^{-1}$, this shows \eqref
  {eq:interpolation}. We emphasize that we obtain a pure
  $L^\infty$-bound on the left hand side of \eqref {eq:interpolation}
  in the case $n=1$.
\end{proof}

\section{Numerical results}
\label{sec.numerics}
In order to illustrate the approximation result of Theorem
\ref{thm:main}, we numerically solve equations \eqref{eq:eps-wave} and
\eqref{eq:weakly-dispersive} in dimensions $n=1$ and $n=2$ with the
initial conditions in \eqref{eq:initial}. We use here a finite
difference method and resolve the solution everywhere; a multi-scale
numerical method that is taylored to the problem at hand was recently
developed, see \cite{MR3090137}.

One of the main practical advantages of the effective equation
\eqref{eq:weakly-dispersive} is its much smaller computational cost
compared to \eqref{eq:eps-wave}. In \eqref{eq:eps-wave} each period of
$a^\eps$ within the computational domain needs to be discretized to
accurately represent the medium. For a fixed domain of $O(1)$ size the
number of periods and hence the number of unknowns scales like
$\eps^{-n}$. On the other hand, for the effective equation
\eqref{eq:weakly-dispersive} the number of unknowns is independent of
$\eps$.

For the spatial discretization of \eqref{eq:eps-wave} we choose the
fourth order finite difference scheme of \cite{CJ96}. In one dimension
($n=1$) and for smooth $a^\varepsilon(x)$ the value of
$\partial_x(a^\varepsilon(x)\partial_x u)$ at the grid point $x=x_j$
is approximated by
\begin{align}
  ({\bf A}^\eps(\lambda)u)_j 
  :=& \frac{4}{3\Delta x}
  \left\{ a^\eps_{j+\frac{1}{2}}\frac{u_{j+1}-u_j}{\Delta x}
    -a^\eps_{j-\frac{1}{2}}\frac{u_{j}-u_{j-1}}{\Delta x}\right\}\\
  & - \frac{1}{6\Delta x}\left\{a^\eps_{j+1}\frac{u_{j+2}-u_j}{2\Delta
      x}-a^\eps_{j-1}\frac{u_{j}-u_{j-2}}{2\Delta x}\right\},
\end{align}
where the coefficients $a^\eps_{j}$ and $a^\eps_{j+\frac{1}{2}}$ are
defined via $a^\eps_{j}= \frac{1}{2\Delta x}
\int_{x_{j-1}}^{x_{j+1}}a^\eps(x)\, dx$ and $\
a^\eps_{j+\frac{1}{2}}=\frac{1}{\Delta x}
\int_{x_{j}}^{x_{j+1}}a^\eps(x)\, dx$, and where $\Delta x$ is the
spacing of the uniform grid $(x_j)_j$.  For the time discretization we
use the standard centered second order scheme resulting in the fully
discrete problem
\begin{equation*}
  u_j^{m+1}=2u_j^m-u_{j}^{m-1}+(\Delta t)^2 ({\bf A}^\eps(\lambda)u^m)_j.
\end{equation*}
In order to initialize the scheme, we set $u_j^0=f(x_j)$ and
approximate $u^1$ via the Taylor expansion $u^1=u^0+\frac{(\Delta
  t)^2}{2}{\bf A}^\eps(\lambda)u^0$.  For the evaluation of ${\bf
  A}^\eps(\lambda)u$ at the boundary of the computational domain we
assume $u=0$ outside the domain. This is legitimate as we choose a
large enough computational domain so that the solution is essentially
zero at the boundary.

The effective equation \eqref{eq:weakly-dispersive} is solved via a
second order centered finite difference scheme. For the second
derivatives we use the standard stencil $({\bf D}_2 w)_j:= (\Delta
x)^{-2}(w_{j+1}-2w_j+w_{j-1})$ and for the fourth derivatives we use
$({\bf D}_4 w)_j:= (\Delta x)^{-4}$ $(w_{j+2}-4w_{j+1}
+6w_j-4w_{j-1}+w_{j-2})$ so that the semidiscrete problem in the case
$n=1$ reads
\begin{equation*}
  \left(({\bf I}-\eps^2E{\bf D}_2)\partial_t^2 u\right)_j
  =\left((A{\bf D}_2 -\eps^2 F {\bf D}_4)u\right)_j.
\end{equation*}
We recall that $E$ and $F$ are scalars when $n=1$. Discretization in time
is performed analogously to the case of equation \eqref{eq:eps-wave}.

The above described methods generalize to $n \geq 2$ dimensions in a
natural way, see \cite{CJ96} for equation \eqref{eq:eps-wave} with
$n=2$.

In general the parameters $a^*, \alpha$, and $\beta$, which determine
the coefficients $A,E$ and $F$ in the effective equation, need to be
computed numerically.  They can be computed by numerically
differentiating the eigenvalue $\mu_0$ as defined in
\eqref{eq:Taylor-mu}.

\subsection{One space dimension}

We choose the material function $a_Y(y) = 1.5+1.4\cos(y)$ and the
initial data $f(x)=e^{-0.4x^2}$ and numerically investigate the
quality of the approximation given by the effective equation. For the
coefficients $A = a^*$ and $C=\alpha$ we find
\begin{align*}
a^*  \approx  0.5385,  \quad  \alpha  \approx -0.5853,
\end{align*}
so that $AD^2=a^*\partial_x^2\approx 0.5385\, \partial_x^2, ED^2 =
-\tfrac{1}{a^*}C \partial_x^2 \approx 1.0869\, \partial_x^2$.

Equation \eqref{eq:eps-wave} was solved with $\Delta x = 2\pi\eps/30$
and $\Delta t = 0.008$ and \eqref{eq:weakly-dispersive} was solved
with $\Delta x \approx 2\pi/100$ and $\Delta t = 0.005$. In
Fig. \ref{F:compare_eps_1D} we plot $u^\eps$ and $w^\eps$ for
$\eps=0.05$ at $t=400=\eps^{-2}$ and for $\eps=0.1$ at
$t=200=2\eps^{-2}$. We see that in both cases the main peak and the
first few dispersive oscillations are well approximated by the
effective model. In the latter case, i.e. with $t$ relatively large
for a given $\eps$, a slight disagreement in the wavelength of the
tail oscillations is visible. Fig. \ref{F:compare_eps_1D} additionally
shows oscillations traveling faster than the main pulse. These
oscillations are physically meaningful as their speed is below the
maximal allowed propagation speed $\hat{c}:=|Y|\int_\R a_Y^{-1/2}(y)
dy$, see \cite{Lamacz-Promotion}, marked by the vertical dotted line.
\begin{figure}[h!]
\begin{center}
 \epsfig{figure = 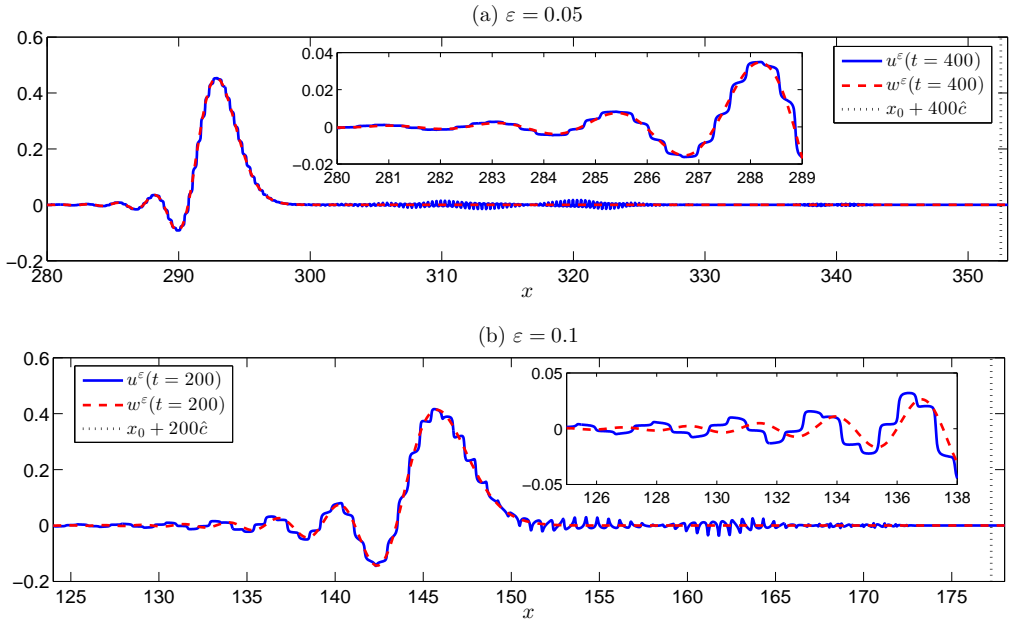,scale=0.50}
 \caption{\label{F:compare_eps_1D} \small One-dimensional equation:
   the solutions $u^\eps$ and $w^\eps$ for $a_Y(y)=1.5+1.4\cos(y)$ and
   $f(x)=e^{-0.4x^2}$ are compared. Only the right propagating part of
   the solution is plotted. In (a) we have $\eps = 0.05$ and in (b)
   $\eps=0.1$. The insets zoom in on the dispersive oscillations to
   the left of the main peak.}
\end{center}
\end{figure}

In Fig. \ref{F:conv_eps_1D} we study the convergence of the
$L^2(\R)-$error for the same material function and initial data as
above. The error is computed at $\eps=0.2, 0.1$ and $0.05$ and
$t=\eps^{-2}$. The error values are approximately $0.1954, 0.0977,
0.0494$. Clearly, the numerical convergence is close to linear, in
agreement with Theorem \ref{thm:main}.
\begin{figure}[h!]
\begin{center}
  \epsfig{figure = 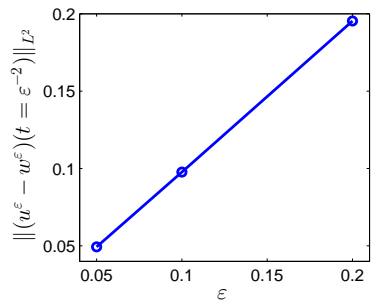,scale=0.52}
  \caption{\label{F:conv_eps_1D} \small Convergence of the $L^2$-error
    $\|u^\eps-w^\eps\|_{L^2}$ at $t=\eps^{-2}$ for
    $a_Y(y)=1.5+1.4\cos(y)$, $f(x)=e^{-0.4x^2}$, and the three values
    $\eps =0.2$, $\eps =0.1$, and $\eps =0.05$. We emphasize that this
    is a severe test for convergence: in both steps, $\eps$ is halved
    and the time instance is quadrupled.}
\end{center}
\end{figure}

\subsection{Two space dimensions}\label{S:num-2D}

Full two-dimensional ($n=2$) simulations for small values of $\eps>0$
and time intervals of order $O(\eps^{-2})$ are computationally
expensive due to the need to discretize each period of size
$O(\eps)\times O(\eps)$ in a domain of size $O(\eps^{-2})\times
O(\eps^{-2})$. We therefore perform instead a simulation that is
designed to mimic the long time behavior of a solution originating
from localized initial data. After a long time the solution develops a
large, close to circular, front. Within the strip
\begin{align*}
  \Omega_s : = x\in \R\times (-\eps\pi,\eps\pi)
\end{align*}
we can expect that the front is nearly periodic in the
$x_2-$direction. Therefore, we perform tests on $\Omega_s$ with
periodic boundary conditions in $x_2$, and initial data that are
localized in $x_1$ and constant in $x_2$. Our choice is to take $f(x)=
e^{-0.6x_1^2}, \ x\in \Omega_s$. We select a material function that
describes a smoothed square structure, namely
\begin{align}\label{eq:a_2d}
  a_Y(y) =& (1+ c(y) - \overline{c})I, \\
  \quad c(y) =&
  \frac{1}{8}\prod_{j=1}^2\left[1+\tanh\left(4(y_j+\tfrac{3}{5}\pi)\right)\right]
  \left[1-\tanh\left(4(y_j-\tfrac{3}{5}\pi)\right)\right],\notag
\end{align}
where $\overline{c}:=\tfrac{1}{|Y|}\int_Yc(y)dy$ and 
$I = \left(\begin{smallmatrix}1 & 0\\ 0 & 1\end{smallmatrix}\right)$. 
This choice ensures a relatively large value of the dispersive
coefficient $\alpha$. We find
\begin{align*}
  a^* \approx 0.5808, \quad \alpha \approx -0.3078, \quad \beta
  \approx 0.0515.
\end{align*}
These values correspond to case 2 in Remark \ref{r:cases} so that
$AD^2=a^*\Delta \approx 0.5808 \,\Delta$,
$ED^2=\tfrac{|\alpha|}{a^*}\Delta\approx 0.5300 \, \Delta,$
$FD^4=(|\alpha|+3\beta)\partial_{x_1}^2\partial_{x_2}^2\approx
0.4623\, \partial_{x_1}^2\partial_{x_2}^2$.  Due to the
$x_2-$independence of the initial data, the solution of the effective
model \eqref{eq:weakly-dispersive} on $\Omega_s$ stays constant in
$x_2$ so that $FD^4$ can be dropped and \eqref{eq:weakly-dispersive}
becomes
\begin{align*}
  \partial_t^2 w^\eps = 0.5808 \, \partial_{x_1}^2w^\eps +\eps^2
  \, 0.53 \, \partial_{x_1}^2\partial_t^2 w^\eps.
\end{align*}
In the simulations of \eqref{eq:eps-wave} we use $\Delta x_1 =\Delta x_2 =
2\pi\eps/30$ and $\Delta t = 0.004$, and in
\eqref{eq:weakly-dispersive} we use $\Delta x_1 =2\pi/100$ and $\Delta t
= 0.01$.

In Fig. \ref{F:2D_plot} the main part of the right propagating half of
the solution $u^\eps$ is plotted for $\eps=0.1$ at
$t=100=\eps^{-2}$. One clearly sees dispersive oscillations behind the
main pulse.
\begin{figure}[h!]
  \begin{center}
    \epsfig{figure = 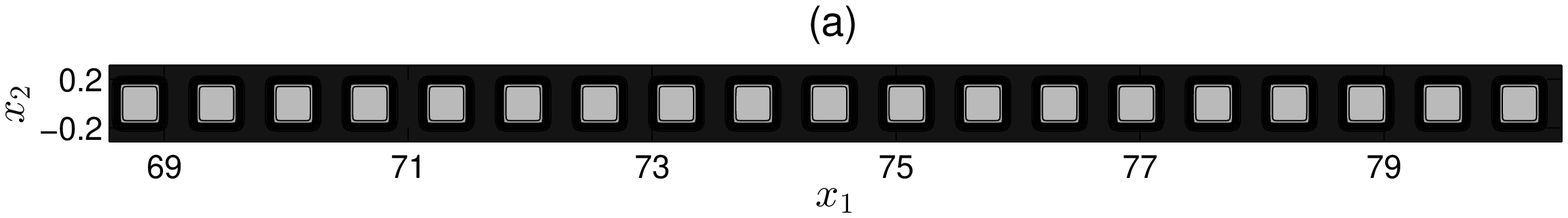,scale=0.45} \epsfig{figure =
      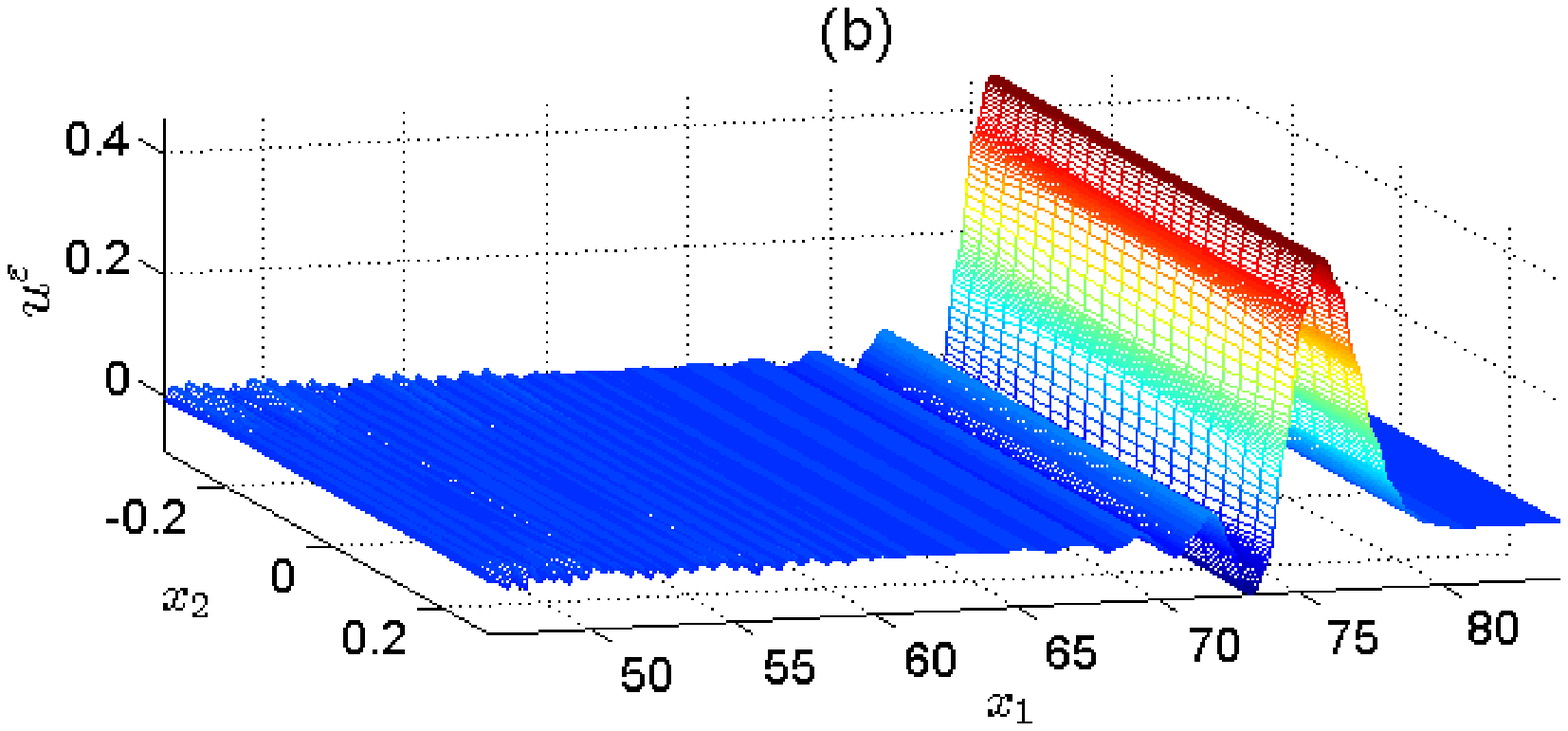,scale=0.53} \epsfig{figure =
      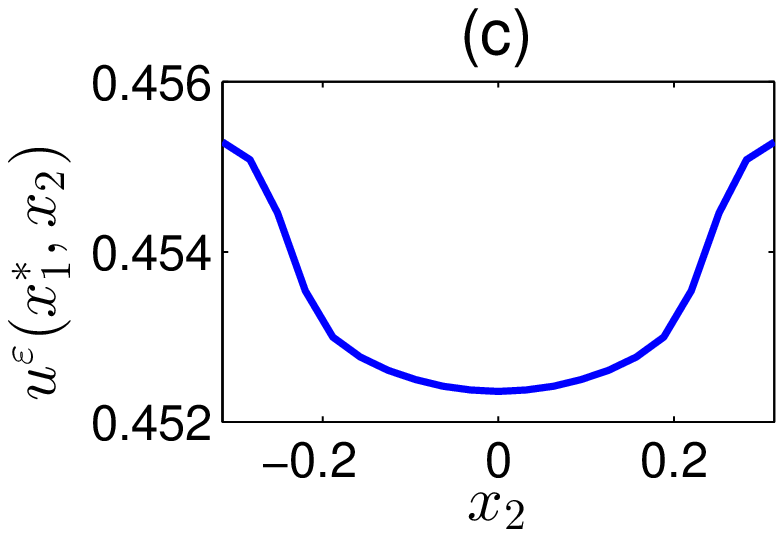,scale=0.44}
    \caption{\label{F:2D_plot} \small Two-dimensional equation: (a)
      The periodic structure $a^\eps(x)$ given by \eqref{eq:a_2d} over
      a section of the strip $\Omega_s$. (b) The main part of the
      right propagating part of the solution $u^\eps$ at $t=100$ for
      $\eps=0.1$ and $f(x)=e^{-0.6x_1^2}$. (c) The $x_2$-profile of
      $u^\eps$ at $x_1=x_1^*$ with $x_1^*$ being the position of the
      peak of the pulse.}
  \end{center}
\end{figure}
Fig.\,\ref{F:compare_2D} shows the agreement between $w^\eps$ and the
$x_2-$mean of $u^\eps$ at $\eps=0.1$ and $t=\eps^{-2}$.
\begin{figure}[h!]
  \begin{center}
    \epsfig{figure = 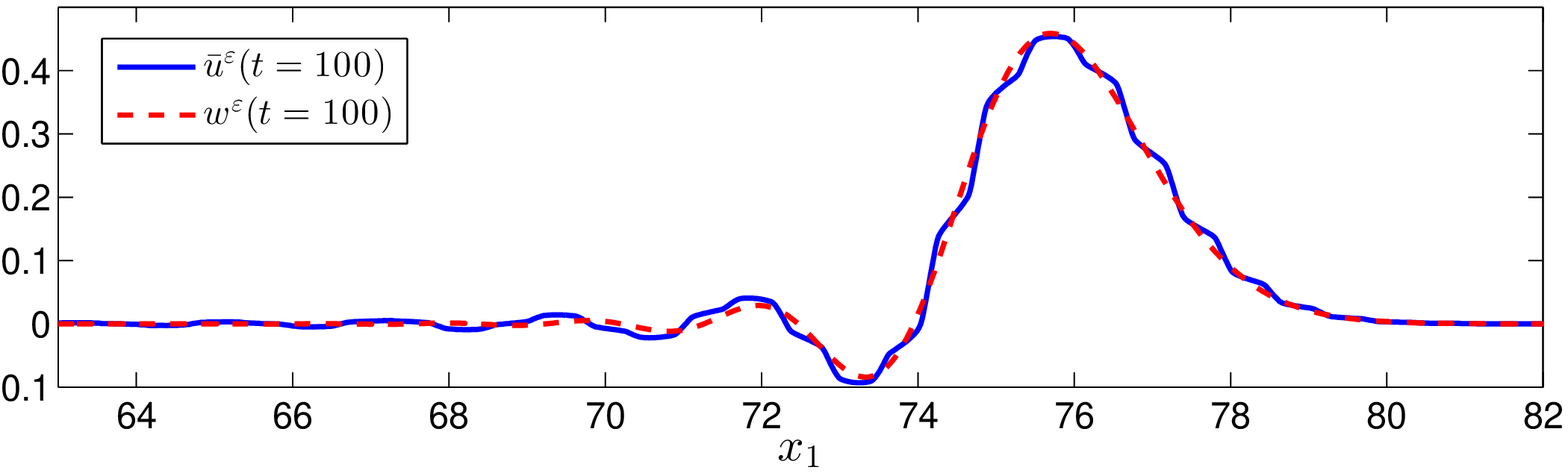,scale=0.6}
    \caption{\label{F:compare_2D} \small Comparison of
      $\bar{u}^\eps(x_1,t)
      := \tfrac{\eps}{2\pi}\int_{-\eps\pi}^{\eps\pi}u^\eps(x_1,x_2,t)dx_2$
      and $w^\eps$ at $\eps=0.1$ and $t=\eps^{-2}$ for $a^\eps(x)$
      given by \eqref{eq:a_2d} and $f(x)=e^{-0.6x^2}$.}
  \end{center}
\end{figure}

\section*{Conclusions}
We have performed an analysis of wave propagation in multi-dimensional
heterogeneous media (periodic with length-scale $\eps>0$). It is
well-known that for large times, solutions cannot be approximated well
by the homogenized second order wave equation. We have provided here a
suitable well-posed dispersive wave equation of fourth order that
describes the original solution $u^\eps$ on time intervals of order
$O(\eps^{-2})$. Our analytical results provide an error estimate of
order $O(\eps)$ between $u^\eps$ and the solution $w^\eps$ of the
dispersive equation. The coefficients of the effective equation are
computable from the dispersion relation, which, in turn, is given by
eigenvalues of a cell-problem. The qualitative agreement between
$u^\eps$ and $w^\eps$ is confirmed by one-dimensional numerical tests,
that even provide a confirmation of the linear convergence of the
error in $\eps$.  In two space dimensions we can observe the validity
of the dispersive equation in a simplified setting, computing
solutions on a long strip.

\appendix

\section{$H^1$-convergence of the Bloch expansion}
\label {app.H1-convergence}

Our aim here is to show that relation \eqref {eq:Bloch-complete-eps}
holds as a convergence of the partial sums in $H^1(\R^n)$.  Since
$\eps>0$ is fixed, for brevity of notation we may as well conclude the
$H^1(\R^n)$-convergence in \eqref {eq:Bloch-complete} for $g\in
H^2(\Omega)$.

With the operator $L := \nabla\cdot (a_Y(y) \nabla)$ we can expand the
two $L^2(\R^n)$-functions $g$ and $h = Lg$ in a Bloch series,
\begin{align*}
  g &= L^2(\R^n)-\lim_{M\to\infty} g^M\quad\text{ for }\quad 
  g^M(y) :=  \sum_{m=0}^M \int_Z \hat g_m(k) w_m(y,k)\, dk\,,\\
  L g = h &= L^2(\R^n)-\lim_{M\to\infty} h^M\quad\text{ for }\quad 
  h^M(y) := \sum_{m=0}^M \int_Z \hat h_m(k) w_m(y,k)\, dk\,.  
\end{align*}
The formulas for $\hat g_m(k)$ and $\hat h_m(k)$ provide, by
construction of $w_m$ as an eigenfunction of $L$ and the symmetry of
$L$,
\begin{align*}
  \hat h_m(k) = \int_{\R^n} (L g)(y) w_m(y,k)^*\, dy
  = \int_{\R^n} g(y) L w_m(y,k)^*\, dy
  =  \mu_m(k) \hat g_m(k)\,.
\end{align*}
In consequence, we obtain
\begin{align*}
  L g^M(y) &= \sum_{m=0}^M \int_Z \hat g_m(k) \mu_m(k) w_m(y,k)\, dk
  = \sum_{m=0}^M \int_Z \hat h_m(k) w_m(y,k)\, dk = h^M(y)\,.
\end{align*}
The right hand side converges in $L^2(\R^n)$ to $h = Lg$. The elliptic
operator $L$ allows to conclude from the $L^2(\R^n)$-convergence $L
g^M \to Lg$ the $H^1(\R^n)$-convergence $g^M\to g$.

\section{Variant of the Gronwall inequality}
\label {app.Gronwall}

We provide now the proof of the Gronwall-type inequality
\eqref{eq:gronwalltype2}.  Let $Y:[0,T]\to [0,\infty)$ be a function
such that, for a constant $Y_0\ge 0$, the relation
\begin{align}
  \label{eq:gronwall3}
  Y(t)\leq 2\int_0^t \|R(., s)\|_{L^2(\R^n)}\sqrt{Y(s)}\,ds + Y_0
\end{align}
holds for all times $t\in [0,T]$. We claim that then
\begin{align} 
  \label{eq:gronwall4}
  Y(t)\leq 2\left(\int_0^t \|R(.,s)\|_{L^2(\R^n)}\,ds\right)^2+2 Y_0
\end{align}
holds for all times $t\in [0,T]$.

\medskip
For the proof we define $Z(t)$ to be the integral on the right hand
side of \eqref{eq:gronwall3},
$$Z(t):=2\int_0^t\|R(., s)\|_{L^2(\R^n)}\sqrt{Y(s)}\,ds.$$ 
Then $Z(0)=0$ and, due to the assumption \eqref{eq:gronwall3},
\begin{align*}
  &\frac{d}{dt}Z(t)=2\|R(., t)\|_{L^2(\R^n)}\sqrt{Y(t)}
  \leq2\|R(., t)\|_{L^2(\R^n)}\sqrt{Z(t)+ Y_0}\,.
\end{align*}
We conclude that
\begin{align*}
  &\frac{d}{dt}\left(\sqrt{Z(t)+Y_0}\right)
  =\left(2\sqrt{Z(t)+Y_0}\right)^{-1}\frac{d}{dt}Z(t)
  \leq\|R(., t)\|_{L^2(\R^n)}.
\end{align*}
Integrating this relation over $[0,t]$ we obtain, recalling $Z(0)=0$,
\begin{align*}
  \sqrt{Z(t)+Y_0}-\sqrt{Y_0}\leq \int_0^t \|R(., s)\|_{L^2(\R^n)}\,ds\,.
\end{align*}
By evaluating the square we find
\begin{align*}
  Z(t)+Y_0 &\leq \left(\sqrt{Y_0} +\int_0^t \|R(.,
    s)\|_{L^2(\R^n)}\,ds \right)^2\\
  &\leq 2Y_0 + 2\left(\int_0^t
    \|R(., s)\|_{L^2(\R^n)}\,ds\right)^2,
\end{align*}
and therefore the claimed result \eqref{eq:gronwall4}, since
$Y(t)\leq Z(t)+Y_0$ holds by assumption.

\bibliographystyle{abbrv}
\bibliography{lit_bloch}

\end{document}